\documentclass[12pt]{amsart}

\usepackage[english]{babel}
\usepackage{amsmath,amssymb,amsfonts,amsthm,amsopn,amstext,amsxtra,amscd}
\usepackage{bm,mathrsfs,mathtools}
\usepackage{cite}
\usepackage{graphicx}
\usepackage{float}
\usepackage{xcolor}
\usepackage{url}

\usepackage[colorlinks,linkcolor=blue,anchorcolor=blue,citecolor=blue,backref=page]{hyperref}
\hypersetup{breaklinks=true}

 \usepackage[norefs,nocites]{refcheck}

\newtheorem{thm}{Theorem}
\newtheorem{lem}[thm]{Lemma}
\newtheorem{lemma}[thm]{Lemma}

\newtheorem{cor}[thm]{Corollary}
\newtheorem{prop}[thm]{Proposition}

\newtheorem{conj}[thm]{Conjecture}
\newtheorem{rem}[thm]{Remark}

\numberwithin{equation}{section}
\numberwithin{thm}{section}
\numberwithin{table}{section}

\usepackage{xcolor}
\usepackage{todonotes}
\newcommand{\abs}[1]{\left|#1\right|}
\newcommand{\norm}[1]{\left\|#1\right\|}
\newcommand{\mand}{\qquad\text{and}\qquad}

\newcommand{\cS}{\mathcal S}

\newcommand{\Z}{\mathbb Z}

\newcommand{\e}{\mathbf e}

\title[Fractional Parts of Sums of Square Roots]
{Moment Estimates and Discrepancy for Sums of Square Roots Modulo One}

\author[Y. Xiao] {Yixiu Xiao}
\address{School of Mathematical Sciences, Shanghai Jiao Tong University, 800 Dongchuan RD, 200240 Shanghai, China}
\email{yixiuxiao98@gmail.com}

\begin{document}

\begin{abstract}
Let $k\ge 2$ be fixed. We study the distribution modulo one of the $n^k$ sums
\begin{equation*}
    \sqrt{a_1} + \cdots + \sqrt{a_k}, \qquad 1\le a_1, \dots, a_k \le n,
\end{equation*}
counted with multiplicity. For
\begin{equation*}
    S(h,n) = \sum_{n/2\le a\le n} \mathbf{e}(h\sqrt{a}), \qquad \mathbf{e}(x) = \exp(2\pi i x),
\end{equation*}
we prove second- and fourth-moment estimates matching the diagonal scale up to a factor $n^\varepsilon$. More precisely,
\begin{equation*}
    \sum_{H/2\le h\le H} \left| S(h,n) \right|^2 \ll_{\varepsilon,\delta} Hn^{1+\varepsilon}
\end{equation*}
uniformly for $H\ge n^{1/2+\delta}$, and
\begin{equation*}
    \sum_{H/2\le h\le H} \left| S(h,n) \right|^4 \ll_{\varepsilon,\delta} Hn^{2+\varepsilon}
\end{equation*}
uniformly for $n^{1/2+\delta} \le H \le n^{2/3}$, where $0<\delta<1/6$ in the fourth-moment estimate. Combining the second-moment bound with pointwise exponential-sum estimates and the Erd\H{o}s--Tur\'an inequality, we obtain
\begin{equation*}
    D_k(n) \le n^{-\rho_k+o(1)}, \qquad \rho_k = \frac{71k+26}{26k+116},
\end{equation*}
as $n\to\infty$, where $D_k(n)$ denotes the discrepancy with respect to arbitrary
subintervals of $[0,1)$.
\end{abstract}



\maketitle

\tableofcontents

\section{Introduction}
\subsection{Set-up and main result}
We are interested in the fractional parts of the sums 
$\sqrt{a_1} + \ldots + \sqrt{a_k}$ with  $k \ge 2$ positive integers 
$a_1, \ldots, a_k \le n$ for a large  $n \ge 1$ (and fixed $k$). 
We refer to~\cite{AnEi, BFMS,ChLi, CMDC, Dub, Iyer, Ste} for a broad variety of results on the distribution of  such sums 
and  their applications, including applications to computer science.

More precisely, for a real number $\xi$ we write
\[
\norm{\xi}=\min_{z\in\Z}|\xi-z|
\]
for its distance to the nearest integer.

Steinerberger~\cite{Ste} gives several upper bounds on 
\begin{equation*}
\begin{aligned}
    \delta_k(n) &= \min \Bigl\{ \lVert \sqrt{a_1} + \cdots + \sqrt{a_k} \rVert : \\
    &\qquad a_i \in \mathbb{Z} \cap [1,n], \; \lVert \sqrt{a_1} + \cdots + \sqrt{a_k} \rVert > 0 \Bigr\}.
\end{aligned}
\end{equation*}
Angluin and Eisenstat~\cite{AnEi} give an explicit construction which shows that
\begin{equation*}
    \delta_2(n) = O \left( n^{-3/2} \right).
\end{equation*}
and a construction of a similar spirit in~\cite{Ste}, which gives 
\begin{equation*}
    \delta_3(n) = O \left( n^{-5/2} \right).
\end{equation*}

We also define the 
inhomogeneous version 
\begin{equation*}
\begin{aligned}
    \Delta_k(n) &= \sup_{\gamma \in [0,1]} \min \Bigl\{ \lVert \sqrt{a_1} + \cdots + \sqrt{a_k} - \gamma \rVert : \\
    &\qquad a_1, \dots, a_k \in \mathbb{Z} \cap [1,n] \Bigr\}.
\end{aligned}
\end{equation*}
Clearly this question is more general as taking $\gamma = 2\Delta_k(n)$, we immediately see that 
$\delta_k(n) \le 3\Delta_k(n)$.

Iyer~\cite[Theorem~1.1]{Iyer} established that
\[
\Delta_k(n)\ll_k n^{-\gamma_k},
\]
where \(\gamma_k\ge (k-1)/4\), with \(\gamma_2=1\) and
\(\gamma_k=k/2\) when \(k=2^m-1\).
Korsky \cite{Korsky} obtained a stronger result for
this inhomogeneous approximation problem. Specialised to square roots,
his theorem gives, for every fixed $k\geq 1$ and every $\varepsilon>0$,
\[
    \Delta_k(n) \ll_{k,\varepsilon} n^{-k/2+\varepsilon},
\]
although with an ineffective implied constant.

In this paper, we address a different, distributional question. We estimate the interval discrepancy $D_k(n)$ of the fractional parts of the points
\begin{equation*}
    \sqrt{a_1} + \cdots + \sqrt{a_k}, \qquad a_1, \dots, a_k \in \mathbb{Z} \cap [1,n].
\end{equation*}
Specifically, for $0 \le \alpha < \beta \le 1$, let $N_k(n;\alpha,\beta)$ denote the number of tuples $(a_1, \dots, a_k) \in [1,n]^k \cap \mathbb{Z}^k$, counted with multiplicity, for which
\begin{equation*}
    \left\{ \sqrt{a_1} + \cdots + \sqrt{a_k} \right\} \in [\alpha, \beta).
\end{equation*}
We define the interval discrepancy
\begin{equation*}
    D_k(n) = \sup_{0 \le \alpha < \beta \le 1} \left| \frac{1}{n^k} N_k(n;\alpha,\beta) - (\beta-\alpha) \right|.
\end{equation*}

The quantities $\delta_k(n)$ and $\Delta_k(n)$ measure the existence of a sum of square roots close to a prescribed target. In contrast, $D_k(n)$ measures the global distribution of all $n^k$ such sums modulo one in arbitrary intervals. Thus discrepancy estimates require cancellation in exponential sums uniformly over many frequencies, and the methods used for inhomogeneous approximation do not directly apply.

The analytic core of the paper consists of second and fourth moment estimates for the exponential sums associated with a dyadic block. Throughout, we use the standard notation $\e(x) := \exp(2\pi i x)$ and write $x \sim X$ to mean $X/2 \le x \le X$.
We define
\begin{equation}
\label{eq:Shn}
    S(h,n)=\sum_{a\sim n}\e(h\sqrt a)
\end{equation}
and, for $r\ge 1$,
\begin{equation}
\label{eq:def_M}
    M_{2r}(H,n)=\sum_{h\sim H}|S(h,n)|^{2r}.
\end{equation}
Our first main theorem gives the expected second-moment bound, up to the factor $n^\varepsilon$.

\begin{thm}[Second moment]
\label{thm:M2-main}
For every fixed $\delta>0$ and every fixed $\varepsilon>0$, uniformly for
\[
    H\ge n^{1/2+\delta},
\]
one has
\[
    M_2(H,n)\ll_{\varepsilon,\delta}Hn^{1+\varepsilon}.
\]
\end{thm}

The same method gives the first higher-moment case in a shorter range.

\begin{thm}[Fourth moment]
\label{thm:M4-main}
For every fixed $0<\delta<1/6$ and every fixed $\varepsilon>0$, uniformly for
\[
    n^{1/2+\delta}\le H\le n^{2/3},
\]
one has
\[
    M_4(H,n)\ll_{\varepsilon,\delta}Hn^{2+\varepsilon}.
\]
\end{thm}

The orders $Hn$ and $Hn^2$ in Theorems~\ref{thm:M2-main} and~\ref{thm:M4-main}, respectively, are the natural diagonal scales for these moments. Combining Theorem~\ref{thm:M2-main} with pointwise exponential-sum bounds and the Erd\H{o}s--Tur\'an inequality yields the following discrepancy estimate.

\begin{thm}[Interval discrepancy]
\label{thm:discrk}
Let $k\ge 2$ be fixed. Then
\[
    D_k(n)\le n^{-\rho_k+o(1)}
\]
as $n\to\infty$, where
\[
    \rho_k=\frac{71k+26}{26k+116}.
\]
\end{thm}

Theorem~\ref{thm:discrk} uses the second-moment estimate, while Theorem~\ref{thm:M4-main} is a further main output of the same smoothing and duality method and represents the first step toward higher moments.
We note that the method of~\cite{Iyer, Korsky} does not apply to bounding the discrepancy $D_k(n)$.

Section~\ref{sec:Pre} collects preliminary and elementary estimates. In Section~\ref{sec:smoothing} we prove Theorems~\ref{thm:M2-main} and~\ref{thm:M4-main} by smoothing the original sums and reducing them to weighted dual quadratic moments. Section~\ref{sec:dis} combines Theorem~\ref{thm:M2-main} with pointwise estimates to prove Theorem~\ref{thm:discrk}.

\subsection{General notation and conventions}  
We use the notations $U=O(V)$, $U\ll V$ and $V\gg U$ interchangeably to mean that
$|U|\leq cV$ for some positive constant $c$. Unless indicated otherwise, implied constants may depend on fixed parameters.

We write $U=V^{o(1)}$ if, for every fixed $\varepsilon>0$, one has
$|U|\leq V^{\varepsilon}$ for all sufficiently large $V$.

For a finite set $\cS$ we use $\#\cS$ to denote its cardinality.

\section{Preliminaries and elementary estimates}
\label{sec:Pre}
 
\subsection{Erd\H{o}s--Tur\'an inequality}

Let $\xi_1, \ldots, \xi_N \in [0,1].$
We recall the following result given in~\cite[Chapter~1, Theorem~1]{Mont}.
\begin{lemma}
\label{lem:Discr}
Let $0 \le \alpha < \beta \le 1$. Then for any integer $H\ge 1$
\begin{equation*}
    D(\alpha, \beta) = \#\{j : \xi_j \in [\alpha, \beta), \; 1\le j \le N\} - (\beta-\alpha)N,
\end{equation*}
we have 
\begin{equation*}
\begin{aligned}
    \left| D(\alpha, \beta) \right| 
    &\le \frac{N}{H+1} \\
    &\quad + 2 \sum_{h=1}^H \left( \frac{1}{H+1} + \min \left\{ \beta - \alpha, \frac{1}{\pi h} \right\} \right) \left| \sum_{j=1}^N \e(h\xi_j) \right|.
\end{aligned}
\end{equation*}
\end{lemma}

For the proof  of Theorem~\ref{thm:discrk}, we use Lemma~\ref{lem:Discr} in the following form, 
which is also known as  Erd\H{o}s--Tur\'{a}n inequality; see, for instance,~\cite[Theorem~1.21]{DrTi}. 

\begin{cor}
\label{cor:ET}
For any integer $H\ge 1$, we have
$$
\sup_{0 \le \alpha < \beta \le 1} \abs{D(\alpha, \beta)}
\ll \frac{N}{H} + \sum_{h=1}^H 
 \frac{1}{h}  
\abs{ \sum_{j=1}^N\e(h\xi_j)}.
$$
\end{cor}

To apply Corollary~\ref{cor:ET} in our setting, we need good bounds for the exponential sums $S(h,n)$ defined in~\eqref{eq:Shn}.

\subsection{Pointwise bounds for exponential sums}

Consider the function
$$
f_h(x) =   h\sqrt{x}
$$
and observe that 
$$
f_h'(x) =  \frac{h}{2 x^{1/2}} \mand f_h''(x)  = - \frac{h}{4 x^{3/2}}.
$$

Since for $|h| \le n^{1/2}$ we have 
$$
\frac{|h|}{4 n^{1/2}}  \le |f_h'(x)| \le \frac{1}{\sqrt{2}} \le 1- \frac{|h|}{4 n^{1/2}}, \quad x \sim n, 
$$
by~\cite[Corollary~8.11]{IwKow} (with $\vartheta  = \frac{1}{4} |h| n^{-1/2}$)  we instantly obtain the following bound on  exponential  sums~\eqref{eq:Shn}. 

\begin{lemma}
\label{lem:vdCorput-1} For any integer $h$ with $1 \le |h| \le n^{1/2}$  we have 
$$
S(h, n)  \ll |h|^{-1} n^{1/2}.
$$
\end{lemma}

To estimate the sums $S(h, n)$ for larger values of $h$ we observe 
that 
$$
\frac{|h|}{4n^{3/2}}\le |f_h''(x)| \le  \frac{|h|}{\sqrt2 n^{3/2}}  \quad x \sim n, 
$$
and appeal to~\cite[Corollary~8.13]{IwKow}  (with $\Lambda = \frac{1}{4} |h| n^{-3/2}$ 
and $\eta  =  2^{3/2}$),  which implies

\begin{lemma}
\label{lem:vdCorput-2} For any integer $h\ne 0$ we have 
$$
S(h, n)  \ll |h|^{1/2} n^{1/4} + |h|^{-1/2} n^{3/4}.
$$
\end{lemma}

We recall that there is a very rich family of bounds of the type of Lemma~\ref{lem:vdCorput-2}, 
appealing to higher derivatives, 
see~\cite{H-B, Rob, RoSa1, RoSa2, RoSa3, Sarg1,Sarg2} for such bounds and~\cite{Bord, Halup}
for their applications, as well  as the references therein. However, these more advanced bounds 
 do not seem to improve our results.

Finally, we also appeal to the theory of exponent pairs. 
Note that in the terminology of~\cite[Chapter~8]{IwKow}, which we use here, 
an exponent pair $(\kappa,\lambda)$ corresponds to the
exponent pair $(\kappa,\lambda+1/2)$ in the 
terminology of~\cite{GrKol,Huxley1,Mont}.

We note that for the  functions $f(x) = h\sqrt{x}$ the condition~\cite[Equation~(8.55)]{IwKow} is  satisfied with $F = h n^{1/2}$, while for~\cite[Equation~(8.56)]{IwKow}  we need $|h| \ge n^{1/2}$.
Hence, the bound~\cite[Equation~(8.58)]{IwKow} yields

\begin{lemma}
\label{lem:ExpPair}
Let $(\kappa,\lambda)$  be an exponent pair.  
Then for any integer $h$ with  $|h| \ge n^{1/2}$ we have
$$
\abs{S(h, n)} \le  |h|^{\kappa + o(1)}n^{1/2+\lambda - \kappa/2} 
$$
where the implied constant depends only on $\kappa$ and $\lambda$.
\end{lemma}

We observe that the statement of Lemma~\ref{lem:ExpPair} is useful only if we have examples of exponent pairs which are better than the trivial pair $(\kappa,\lambda) = (0,1/2)$. However, such non-trivial pairs are indeed well-known, and we refer to an outline of the most significant results given by Bourgain~\cite{Bour}, as well as to the new exponent pair
\begin{equation}
\label{eq:BourgPair}
    (\kappa, \lambda) = \left( \frac{13}{84}, \frac{13}{84} \right)
\end{equation} 
given in~\cite[Theorem~6]{Bour}.

We now make a simplifying assumption that we are given an exponent pair with 
$\kappa = \lambda$, which actually holds for many of the most interesting exponent pairs, 
including~\eqref{eq:BourgPair}.

Combining Lemmas \ref{lem:vdCorput-1}, \ref{lem:vdCorput-2} and~\ref{lem:ExpPair}, we shall use the following convenient upper envelope.

Let $(\kappa, \kappa)$ be an exponent pair. Then for any non-zero integer $h$ we have 
\begin{equation}
\label{eq:OptBound}
\abs{S(h, n)} \le
 \begin{cases} 
 |h|^{-1} n^{1/2}& \text{if}\ |h| \le n^{1/2},\\
  |h|^{1/2} n^{1/4}& \text{if}\   n^{1/2} <  |h|\le n^{(1+2 \kappa)/(2(1- 2\kappa))}, \\ 
|h|^{\kappa + o(1)}n^{1/2+\kappa/2}& \text{if}\    |h| >  n^{(1+2 \kappa)/(2(1-2\kappa))}.
\end{cases}
\end{equation}

Substituting the exponent pair~\eqref{eq:BourgPair} in Lemma~\ref{lem:ExpPair} we see
that for any  $h$ with  $|h| \ge n^{1/2}$ we have
$$
\abs{S(h, n)} \le |h|^{13/84 + o(1)}n^{97/168 } . 
$$

\subsection{An elementary second moment bound}
In this subsection, we establish an elementary bound for the second moment $M_2(H,n)$ defined in~\eqref{eq:def_M}.

We first record two auxiliary lemmas.

\begin{lemma}
Let $q\ge 1$ and $K\in\mathbb Z$. Define
\begin{equation}
\label{eq:def_rho}
    \rho_q(K)
 =
 \#\{x\bmod q: x^2\equiv K \pmod q\}.
\end{equation}
Then, for every fixed $\varepsilon>0$,
\begin{equation}
\label{eq:pho_q}
\rho_q(K)\ll_\varepsilon q^\varepsilon (K,q)^{1/2}.
\end{equation} 
Moreover, for every $R\ge 0$,
\begin{equation}
\label{eq:sum_pho}
    \sum_{0<|K|\le R}\rho_q(K)
 \ll_\varepsilon q^\varepsilon(1+R).
\end{equation}
\end{lemma}

\begin{proof}
We first prove the pointwise bound. Write
\[
 q=\prod_{p^\nu\Vert q}p^\nu .
\]
By the Chinese remainder theorem,
\[
 \rho_q(K)=\prod_{p^\nu\Vert q}\rho_{p^\nu}(K).
\]
Thus it is enough to estimate $\rho_{p^\nu}(K)$.
Fix $p^\nu\Vert q$, and put
\[
 t=\min\{v_p(K),\nu\}.
\]

First suppose that $K\equiv 0\pmod {p^\nu}$. Then
\[
 x^2\equiv 0\pmod {p^\nu}
\]
is equivalent to
\[
 v_p(x)\ge \left\lceil \frac{\nu}{2}\right\rceil.
\]
Hence the number of solutions $x \bmod p^\nu$ is
\[
 p^{\nu-\lceil \nu/2\rceil}
 =
 p^{\lfloor \nu/2\rfloor}
 \le p^{\nu/2}.
\]
Since in this case $(K,p^\nu)=p^\nu$, we get
\begin{equation}
\label{eq:rho_p_nu_1}
    \rho_{p^\nu}(K)\le (K,p^\nu)^{1/2}.
\end{equation}

Now suppose that $K\not\equiv 0\pmod {p^\nu}$. Then
$t=v_p(K)<\nu$. If
\[
 x^2\equiv K\pmod {p^\nu},
\]
then $v_p(x^2)=v_p(K)=t$, and therefore $t$ must be even.  
Thus we may assume that
\[
 t=2s.
\]
Write
\[
 K=p^{2s}K_0,\qquad p\nmid K_0.
\]
Every solution must have the form
\[
 x=p^s y,\qquad p\nmid y.
\]
Substituting into the congruence gives
\[
 p^{2s}y^2\equiv p^{2s}K_0\pmod {p^\nu},
\]
or equivalently
\[
 y^2\equiv K_0\pmod {p^{\nu-2s}}.
\]
Since $K_0$ is coprime to $p$, the congruence 
$y^2\equiv K_0\pmod {p^r}$
 has at most 2 solutions for odd $ p$
 and at most  4  solutions for $p=2$. Each such solution gives $p^s$ possibilities
for $x\bmod p^\nu$. Hence
\[
 \rho_{p^\nu}(K)\ll p^s=p^{t/2}.
\]
Since $p^t=(K,p^\nu)$, this gives
\begin{equation}
\label{eq:rho_p_nu_2}
    \rho_{p^\nu}(K)\ll (K,p^\nu)^{1/2}.
\end{equation}

Combining \eqref{eq:rho_p_nu_1} and \eqref{eq:rho_p_nu_2}, we have
\[
 \rho_{p^\nu}(K)\ll (K,p^\nu)^{1/2},
\]
with an absolute implied constant. Multiplying over all prime powers
dividing $q$, we obtain
\[
 \rho_q(K)
 \ll
 C^{\omega(q)}
 \prod_{p^\nu\Vert q}(K,p^\nu)^{1/2}
 =
 C^{\omega(q)}(K,q)^{1/2}.
\]
Since
\[
 C^{\omega(q)}\ll_\varepsilon q^\varepsilon,
\]
the \eqref{eq:pho_q} follows.

It remains to prove \eqref{eq:sum_pho}. By \eqref{eq:pho_q},
\[
 \sum_{0<|K|\le R}\rho_q(K)
 \ll_\varepsilon
 q^\varepsilon
 \sum_{0<|K|\le R}(K,q)^{1/2}.
\]
If $R<1$, the sum is empty. Thus we may assume $R\ge 1$. Since
\[
 (K,q)^{1/2}
 \le
 \sum_{\substack{d\mid q\\ d\mid K}}d^{1/2},
\]
we get
\[
 \sum_{0<|K|\le R}(K,q)^{1/2}
 \le
 \sum_{d\mid q}d^{1/2}
 \#\{0<|K|\le R:d\mid K\}.
\]
But
\[
 \#\{0<|K|\le R:d\mid K\}\ll \frac{R}{d}.
\]
Therefore
\[
 \sum_{0<|K|\le R}(K,q)^{1/2}
 \ll
 R\sum_{d\mid q}d^{-1/2}
 \ll_\varepsilon Rq^\varepsilon.
\]
Absorbing the extra factor $q^\varepsilon$ into the implied constant
by replacing $\varepsilon$ with $\varepsilon/2$, we obtain
\[
 \sum_{0<|K|\le R}\rho_q(K)
 \ll_\varepsilon q^\varepsilon R
 \ll_\varepsilon q^\varepsilon(1+R).
\]
This completes the proof.
\end{proof}

\begin{lemma}[Near-collisions of square roots]
\label{lem:near-collision-sqrt}
Let $0<\eta\le 1/2$, and put
\[
 A_2(n,\eta)
 =
 \#\left\{(a,b): a,b\sim n,\ 
 \left\|\sqrt a-\sqrt b\right\|\le \eta\right\}.
\]
Then, for every fixed $\varepsilon>0$,
\[
 A_2(n,\eta)\ll_\varepsilon n^\varepsilon\left(n+n^2\eta\right).
\]
\end{lemma}

\begin{proof}
We first dispose of the exact integer collisions. Suppose that
\[
 \sqrt a-\sqrt b\in \mathbb Z.
\]
If $\sqrt a-\sqrt b=0$, then $a=b$, which gives $O(n)$ pairs.
If $\sqrt a-\sqrt b=m\neq 0$, then
\[
 a=b+m^2+2m\sqrt b,
\]
and hence $\sqrt b\in \mathbb Q$. Since $b$ is an integer, this
implies that $b$ is a square, and then $a$ is also a square. Thus
$a=s^2$, $b=t^2$, with $s,t\ll \sqrt n$, which gives $O(n)$
such pairs. Therefore the contribution of exact integer collisions is $O(n).$

It remains to count the pairs for which
\[
 0<\left\|\sqrt a-\sqrt b\right\|\le \eta.
\]
Choose an integer $m$ and a real number $\xi$ such that
\[
 \sqrt a-\sqrt b=m+\xi,\qquad 0<|\xi|\le \eta.
\]
Since $a,b\sim n$, we have $|m|\ll \sqrt n$. By symmetry between
$a$ and $b$, it is enough to consider $m\ge 0$.

If $m=0$, then
\[
 |\sqrt a-\sqrt b|\le \eta
\]
implies
\[
 |a-b|\le \eta(\sqrt a+\sqrt b)\ll \eta\sqrt n.
\]
Hence the number of such pairs is
\begin{equation}
\label{eq:m=0}
     \ll n(1+\eta\sqrt n)\ll n+n^2\eta.
\end{equation}

Now assume that $1\le m\ll \sqrt n$. Define
\[
 u=a-b-m^2.
\]
Then
\[
 a-(\sqrt b+m)^2
 =
 u-2m\sqrt b
 =
 (\sqrt a-\sqrt b-m)(\sqrt a+\sqrt b+m),
\]
and therefore
\[
 |u-2m\sqrt b|\ll \eta\sqrt n.
\]
In particular, $u$ is restricted to an interval of length $O(m\sqrt n)$, and
\[
 |u|\ll m\sqrt n.
\]

Put
\[
 K=u^2-4m^2b.
\]
Note that
$K\ne 0$, since $K=0$ would imply $u=2m\sqrt b$, and hence $\sqrt a-\sqrt b=m$, which belongs to the
exact integer collision case already discussed.

\[
 |K|
 =
 |u-2m\sqrt b|\,|u+2m\sqrt b|
 \ll \eta mn.
\]
Moreover,
\[
 u^2\equiv K \pmod {4m^2}.
\]
Conversely, once $m,u,K$ are fixed and
\[
 u^2\equiv K \pmod {4m^2},
\]
the integer $b$, if it exists, is determined by
\[
 b=\frac{u^2-K}{4m^2},
\]
and then $a=b+m^2+u$ is also determined.

For a fixed $m\ge 1$, let $\mathcal U_m$ denote the set of possible integer
values of $u$. As observed above, $\mathcal U_m$ is contained in an 
interval of length $O(m\sqrt n)$. 
Therefore the number $N_m$ of admissible pairs corresponding to this fixed  $m$ is at most
\[
N_m
 \le
\sum_{0<|K|\le C\eta mn}
\#\{u\in\mathcal U_m: u^2\equiv K \pmod {4m^2}\}.
\]
For a fixed $K$, 
\begin{equation*}
    \begin{split}
\#\{u\in\mathcal U_m: u^2\equiv K \pmod {4m^2}\}
 &\ll
\left(1+\frac{|\mathcal U_m|}{4m^2}\right)\rho_{4m^2}(K)\\
 &\ll
\left(1+\frac{\sqrt n}{m}\right)\rho_{4m^2}(K),
    \end{split}
\end{equation*}
where $\rho$ is defined in \eqref{eq:def_rho}.
Hence, invoking \eqref{eq:sum_pho} with
$q=4m^2$ and $R=C\eta mn$, we obtain
\begin{equation}
\label{eq:Nm}
    \begin{split}
      N_m
 &\ll_\varepsilon
\left(1+\frac{\sqrt n}{m}\right)
\sum_{0<|K|\le C\eta mn}\rho_{4m^2}(K)\\
 &\ll_\varepsilon
m^\varepsilon
\left(1+\frac{\sqrt n}{m}\right)
(1+\eta mn).  
    \end{split}
\end{equation}

Summing \eqref{eq:Nm} over $1\le m\ll \sqrt n$, we get
\[
\sum_{1\le m\ll \sqrt n} N_m
\ll_\varepsilon
\sum_{1\le m\ll \sqrt n}
m^\varepsilon
\left(1+\frac{\sqrt n}{m}\right)(1+\eta mn)
\ll_\varepsilon n^\varepsilon(n+n^2\eta).
\]
Combining this with the $m=0$ contribution \eqref{eq:m=0} and with the exact
integer collisions gives
\[
 A_2(n,\eta)\ll_\varepsilon n^\varepsilon(n+n^2\eta),
\]
as required.
\end{proof}

\begin{lemma}[Second moment]
\label{lem:M2Hn}
For every $H\ge 2$ and every fixed $\varepsilon>0$, we have
\[
 M_2(H,n)
 \ll_\varepsilon
n^\varepsilon\left(Hn+n^2\log(2H)\right)
\]
where 
$M_2(H,n)$ is given by \eqref{eq:def_M}.
Equivalently,
\[
 M_2(H,n)\ll (Hn+n^2 \log(2H))n^{o(1)}.
\]
\end{lemma}

\begin{rem}
    This elementary estimate is most useful when $H \gg n$. In the range $n^{1/2+\delta}\ll H\ll n$, it will be superseded by the smoothing method of Section \ref{sec:smoothing}.
\end{rem}

\begin{proof}
The case $H<2$ is trivial, so we assume $H\ge 2$. Squaring out and
interchanging the order of summation, we get
\begin{equation}
\label{eq:MHn2}
    M_2(H,n) = \sum_{a,b \sim n} \sum_{h \sim H} \e \left( h (\sqrt{a} - \sqrt{b}) \right).
\end{equation}

Let $\eta_0=H^{-1}$, let $\eta_i=e^i/H$ for $1\le i<I$, where
$\eta_{I-1}<1/2\le e^I/H$, and set $\eta_I=1/2$.

By the well-known bound 
\begin{equation*}
    \left| \sum_{h \sim H} \e(\alpha h) \right| \ll \min\left\{H, \frac{1}{\lVert \alpha \rVert}\right\},
\end{equation*}
see~\cite[Equation~(8.6)]{IwKow}, we derive from~\eqref{eq:MHn2} that 
\begin{equation*}
\begin{split}
    M_2(H,n) 
    &\ll \sum_{\substack{a,b \sim n \\ \lVert \sqrt{a} - \sqrt{b} \rVert \le \eta_0 }} H + \sum_{i=1}^I \sum_{\substack{a,b \sim n \\ \eta_{i-1} \le \lVert \sqrt{a} - \sqrt{b} \rVert \le \eta_i }} \eta_{i-1}^{-1} \\
    &\le A_2(n, \eta_0) H + \sum_{i=1}^I A_2(n, \eta_i) \eta_{i-1}^{-1}.
\end{split}
\end{equation*}

Using Lemma \ref{lem:near-collision-sqrt}, we now derive
\[
 M_2(H,n)
 \ll_\varepsilon
n^\varepsilon\left(Hn+n^2\log H\right),
\]
which concludes the proof. 
\end{proof}

\subsection{A conjectural counting approach for higher moments}
We record a natural higher-dimensional analogue of Lemma~\ref{lem:near-collision-sqrt}. For $0<\eta\le 1/2$, define
\begin{equation*}
\begin{aligned}
    A_{2r}(n,\eta) := \# \Biggl\{ & (a_1, \dots, a_r, b_1, \dots, b_r) : a_i, b_j \sim n, \\
    & \left\lVert \sum_{i=1}^r \sqrt{a_i} - \sum_{j=1}^r \sqrt{b_j} \right\rVert \le \eta \Biggr\}.
\end{aligned}
\end{equation*}

Expanding $M_{2r}(H,n)$ in \eqref{eq:def_M}, we obtain
$$
    M_{2r}(H,n) = \sum_{\substack{a_1,\ldots,a_r\sim n\\ b_1,\ldots,b_r\sim n}} \sum_{h\sim H} \e\left( h\left( \sum_{i=1}^r \sqrt{a_i} - \sum_{j=1}^r \sqrt{b_j} \right) \right).
$$
Using
$$
    \left|\sum_{h\sim H}\e(h\alpha)\right| \ll \min\{H,\|\alpha\|^{-1}\},
$$
and arguing as in the proof of Lemma~\ref{lem:M2Hn}, we get the following reduction. Put
$$
    \eta_0=H^{-1},
$$
and let $\eta_i=e^i/H$ for $1\le i<I$, where $I$ is chosen so that
$$
    \eta_{I-1}<\frac12\le \frac{e^I}{H},
$$
and finally set $\eta_I=1/2$. Then
\begin{equation}
\label{eq:M2rviaA2r}
    M_{2r}(H,n) \ll H A_{2r}(n,H^{-1}) + \sum_{i=1}^{I} A_{2r}(n,\eta_i)\eta_{i-1}^{-1}.
\end{equation}

Thus higher even moments are reduced to bounding $A_{2r}(n,\eta)$. If the fractional parts of
$$
    \sum_{i=1}^r \sqrt{a_i}, \qquad a_i\sim n,
$$
behaved like $n^r$ random points in $[0,1)$, then the expected number of pairs at distance at most $\eta$ would be of order
$$
    n^r+n^{2r}\eta.
$$
This leads to the following conjecture.

\begin{conj}
\label{conj:A2r}
For every fixed integer $r\ge 1$ and every fixed $\varepsilon>0$,
uniformly for $0<\eta\le 1/2$, one has
\begin{equation}
\label{eq:A2r}
    A_{2r}(n,\eta)
    \ll_{r,\varepsilon}
    n^{r+\varepsilon}+n^{2r+\varepsilon}\eta .
\end{equation}
\end{conj}

Assuming Conjecture \ref{conj:A2r}, the reduction
\eqref{eq:M2rviaA2r} gives
\begin{equation}
\label{eq:M2rbdviaA2r}
    M_{2r}(H,n)
    \ll_{r,\varepsilon}
    Hn^{r+\varepsilon}
    +n^{2r+\varepsilon}\log(2H).
\end{equation}

\begin{rem}
The conjectural estimate \eqref{eq:A2r} is expected to be sharp up to
the factor $n^\varepsilon$. The term $n^r$ is forced by diagonal
solutions, for example by taking $(b_1,\ldots,b_r)$ to be a permutation
of $(a_1,\ldots,a_r)$. The term $n^{2r}\eta$ is also unavoidable at
the scale $\eta\asymp 1$.

However, the moment estimate \eqref{eq:M2rbdviaA2r} is not always the
strongest available bound. In particular, for $r\ge 2$ and $H\ll n$,
the pointwise bounds \eqref{eq:OptBound} may give
stronger estimates. This motivates the smoothing method developed in
Section \ref{sec:smoothing}.
\end{rem}

\section{Smoothing and dual moment estimates}
\label{sec:smoothing}

\subsection{Outline of the smoothing method}

The purpose of this section is to prove the moment estimates stated in Theorems~\ref{thm:M2-main} and~\ref{thm:M4-main}. We develop a smoothing method for
\[
    M_{2r}(H,n)=\sum_{h\sim H}|S(h,n)|^{2r}
\]
in the range $H\ge n^{1/2+\delta}$. We apply it unconditionally to $r=1$ and $r=2$. For higher moments, the same argument leads to a natural dual problem, which is discussed at the end of the section.

 Lemma~\ref{lem:M2Hn} is
sufficient when $H\gg n$, but it is not strong enough in the
intermediate range
\[
    n^{1/2+\delta}\ll H\ll n.
\]
The second-moment argument below treats this intermediate range and, together with Lemma~\ref{lem:M2Hn}, proves Theorem~\ref{thm:M2-main}. The fourth-moment argument proves Theorem~\ref{thm:M4-main} in the shorter range $H\le n^{2/3}$.

We now outline the mechanism of the proof. We replace the sharp sum
$S(h,n)$ by a smoothed sum $S_\omega(h,n)$ introduced in \eqref{eq:Swhn}, where the smoothing is
confined to two endpoint intervals of length $O(\omega)$. Applying
Poisson summation and stationary phase transforms $S_\omega(h,n)$ into
a dual quadratic sum of the form
\[
    S_\omega(h,n)
    =
    CB
    \sum_{m\asymp H/\sqrt n}
    V_{\omega;h,m}
    e\left(\frac{h^2}{4m}\right)
    +
    \text{acceptable error},
\]
with 
$$
B=\frac{H}{(H/\sqrt n)^{3/2}}.
$$
Thus the original moment problem is reduced to estimating weighted
dual moments for the phases
\[
    e\left(\frac{h^2}{4m}\right).
\]
The weights $V_{\omega;h,m}$ have uniformly bounded variation in the
$h$-aspect, and can therefore be removed by partial summation. In the
second moment case, the resulting dual sums are estimated by a
quadratic Weyl bound and an elementary denominator count. In the fourth
moment case, one also needs a reciprocal-energy estimate and a
reciprocal-denominator gcd sum.

\subsection{Smooth cut-offs and Poisson summation}

We begin by defining a smooth cut-off function with an adjustable edge length.

\begin{lemma}[Smooth cut-off with adjustable edge length] 
\label{lem:w def}
Let $n\geq 1$ and $1\leq \omega \leq n/10$. There exists a function
$$
W_{\omega,n}\in C^\infty(\mathbb R),\qquad \operatorname{supp}W_{\omega,n}\subset [n/2,n],
$$
such that
\[
        0\leq W_{\omega,n}\leq 1,
\]
\[
        W_{\omega,n}(x)=1
        \qquad
        \left(\frac n2+\omega\leq x\leq n-\omega\right),
\]
and, for every integer $j\geq 0$,
\[
        W_{\omega,n}^{(j)}(x)\ll_j \omega^{-j}.
\]
Moreover, for every $j \ge 1$,
\begin{equation*}
    \left\lVert W_{\omega,n}^{(j)} \right\rVert_1 \ll_{j} \omega^{1-j}.
\end{equation*}

\end{lemma}

\begin{proof}
Let
\[
\varphi(t)=
\begin{cases}
e^{-1/t}, & t>0,\\
0, & t\leq 0,
\end{cases}
\]
and define
\[
        \psi(t)=
        \frac{\varphi(t)}
        {\varphi(t)+\varphi(1-t)} .
\]
Then $\psi\in C^\infty(\mathbb R)$, $0\leq \psi\leq 1$,
\[
        \psi(t)=0 \qquad (t\leq 0),
\]
and
\[
        \psi(t)=1 \qquad (t\geq 1).
\]
Now set
\[
        W_{\omega,n}(x)
        =
        \psi\left(\frac{x-n/2}{\omega}\right)
        \psi\left(\frac{n-x}{\omega}\right).
\]

We first verify the support and plateau properties. If $x\leq n/2$,
then $(x-n/2)/\omega \leq 0$, so the first factor is zero. If $x\geq n$,
then $(n-x)/\omega \leq 0$, so the second factor is zero. Hence
\[
        W_{\omega,n}(x)=0
        \qquad
        (x\leq n/2 \text{ or } x\geq n),
\]
and therefore
\[
        \operatorname{supp} W_{\omega,n}\subset [n/2,n].
\]
Since $\psi$ is smooth and the two transition factors vanish to infinite
order at the endpoints, we have
\[
        W_{\omega,n}\in C_c^\infty(\mathbb R).
\]
Also, since $0\leq \psi\leq 1$, we have
\[
        0\leq W_{\omega,n}\leq 1.
\]

If
\[
        \frac n2+\omega \leq x \leq n-\omega,
\]
then
\[
        \frac{x-n/2}{\omega}\geq 1,
        \qquad
        \frac{n-x}{\omega}\geq 1.
\]
Thus both factors are equal to $1$, and hence
\[
        W_{\omega,n}(x)=1
        \qquad
        \left(\frac n2+\omega\leq x\leq n-\omega\right).
\]

It remains to prove the derivative bounds. For every integer $r\geq 0$,
the chain rule gives
\[
        \frac{d^r}{dx^r}
        \psi\left(\frac{x-n/2}{\omega}\right)
        =
        \omega^{-r}
        \psi^{(r)}\left(\frac{x-n/2}{\omega}\right),
\]
and
\[
        \frac{d^r}{dx^r}
        \psi\left(\frac{n-x}{\omega}\right)
        =
        (-1)^r \omega^{-r}
        \psi^{(r)}\left(\frac{n-x}{\omega}\right).
\]
Since every derivative of $\psi$ is bounded, we obtain
\[
        \left\|
        \frac{d^r}{dx^r}
        \psi\left(\frac{x-n/2}{\omega}\right)
        \right\|_\infty
        \ll_r \omega^{-r},
\]
and similarly
\[
        \left\|
        \frac{d^r}{dx^r}
        \psi\left(\frac{n-x}{\omega}\right)
        \right\|_\infty
        \ll_r \omega^{-r}.
\]
By the product rule,
\[
        W_{\omega,n}^{(j)}(x)
        =
        \sum_{r=0}^j
        \binom jr
        \frac{d^r}{dx^r}
        \psi\left(\frac{x-n/2}{\omega}\right)
        \frac{d^{j-r}}{dx^{j-r}}
        \psi\left(\frac{n-x}{\omega}\right).
\]
Therefore
\[
        \|W_{\omega,n}^{(j)}\|_\infty
        \ll_j
        \sum_{r=0}^j
        \omega^{-r} \omega^{-(j-r)}
        \ll_j
         \omega^{-j}.
\]

Since $W_{\omega,n}^{(j)}$ is supported in the two transition intervals, whose total length is $O(\omega)$, the $L^1$-bound follows from the supremum bound.

This completes the proof.
\end{proof}

We now fix a parameter $1\leq \omega \leq n/10$, and let $W_{\omega,n}$
be the smooth cut-off constructed in Lemma~\ref{lem:w def}.
Define
\begin{equation}
\label{eq:Swhn}
  S_\omega(h,n)
  =
  \sum_{a\in\mathbb Z}W_{\omega,n}(a)e(h\sqrt a),
\end{equation}
and for $r \ge 1,$
\begin{equation*}
  M_{2r,\omega}(H,n)
  =
  \sum_{h\sim H}|S_\omega(h,n)|^{2r}.
\end{equation*}
The sharp and smoothed cut-offs differ only on two endpoint intervals of
total length $O(\omega)$. Therefore
\begin{equation*}
  S(h,n)-S_\omega(h,n)\ll \omega+1.
\end{equation*}
Since $\omega\ge 1$, we obtain
\begin{equation*}
  |S(h,n)|^{2r}
  \ll_r
  |S_\omega(h,n)|^{2r}+\omega^{2r},
\end{equation*}
and hence
\begin{equation}
  \label{eq:sharp-to-smooth-higher-moment}
  M_{2r}(H,n)
  \ll_r
  M_{2r,\omega}(H,n)+H\omega^{2r}.
\end{equation}
The term $H\omega^{2r}$ is the cost of removing the smoothing and should be
balanced only at the end.

Applying Poisson summation to the function
\[
  x\mapsto W_{\omega,n}(x)e(h\sqrt x),
\]
we obtain
\begin{equation}
\label{eq:S_sum_I}
     S_{\omega}(h,n)=\sum_{m\in\mathbb Z} I_m(h,n),
\end{equation}
where
\[
  I_m(h,n):=
  \int_{\mathbb R} W_{\omega,n}(x)e(h\sqrt x-mx)\,dx,
  \qquad m\in\mathbb Z.
\]
We now make the change of variables $x=y^2$.  Since
$W_{\omega,n}$ is supported in $[n/2,n]$, only positive $y$ contributes,
and therefore
\begin{equation*}
  I_m(h,n)
  =
  \int_{\mathbb R} 2y W_{\omega,n}(y^2)
  e(hy-my^2)\,dy .
\end{equation*}
Put
\begin{equation}
    \label{eq:A}
  A_{\omega,n}(y)=2yW_{\omega,n}(y^2)\mathbf 1_{(0,\infty)}(y).
\end{equation}
Then
\[
  I_m(h,n)
  =
  \int_{\mathbb R} A_{\omega,n}(y)e(hy-my^2)\,dy.
\]
For $m>0$, let
\[
  y_m:=\frac{h}{2m}.
\]
Completing the square gives
\begin{equation*}
  I_m(h,n)
  =
  e\!\left(\frac{h^2}{4m}\right)
  \int_{\mathbb R}
  A_{\omega,n}(y)e(-m(y-y_m)^2)\,dy .
\end{equation*}
The support of $A_{\omega,n}$ is contained in
\[
  \sqrt{n/2}\leq y\leq \sqrt n.
\]
Moreover, the plateau of $W_{\omega,n}$ gives
\[
  W_{\omega,n}(y^2)=1
  \qquad
  \left(n/2+\omega \leq y^2\leq n-\omega\right).
\]
Thus the stationary point $y_m$ can contribute only when
\[
  y_m\asymp \sqrt n,
\]
or equivalently
\[
  m\asymp \frac{h}{\sqrt n}.
\]

\subsection{Stationary and non-stationary phase estimates}

We now record the stationary phase estimate in the range
\[
  m\asymp \frac{h}{\sqrt n}.
\]

\begin{lemma}[Stationary phase with adjustable edge length]
\label{lem:Im-stationary}
Suppose $n$ is sufficiently large and $1 \leq \omega \leq n/10$, and let $W_{\omega,n}$ be as in
Lemma~\ref{lem:w def}. $A_{\omega,n}(y)$
is given by \eqref{eq:A},
and
\[
        I_m(h,n)=\int_{\mathbb R}A_{\omega,n}(y)e(hy-my^2)\,dy .
\]
Let $m>0$, and set
\[
        y_m=\frac{h}{2m}.
\]
Assume that
\begin{equation}
\label{eq:stationary range}
    \frac{1}{2}\sqrt{\frac n2}<y_m< \frac{3}{2}\sqrt n ,\ \text{ i.e. } m \in \left[\frac{1}{3}\frac{h}{\sqrt{n}}, \sqrt{2}\frac{h}{\sqrt{n}}\right].
\end{equation}
Then
\[
        I_m(h,n)
        =        \frac{e\!\left(\frac{h^2}{4m}-\frac{1}{8}\right)}{\sqrt{2m}}\,
        A_{\omega,n}(y_m)
        +O\left(      
            \frac{n^{3/2}}{m^{3/2} \omega^2}
        \right),
\]
where the implied constant is absolute. Equivalently,
\[
        I_m(h,n)
        =        \frac{e\!\left(\frac{h^2}{4m}-\frac{1}{8}\right)}{\sqrt{2m}}\,
        \frac{h}{m}
        W_{\omega,n}\!\left(\frac{h^2}{4m^2}\right)
        +O\left(
            \frac{n^{3/2}}{m^{3/2} \omega^2}
        \right).
\]
In particular, if
$$
\sqrt{\frac{n}{2}+\omega} \leq y_m \leq \sqrt{n-\omega},
$$
then $W_{\omega,n}(y_m^2)=1$, and therefore
\[
        I_m(h,n)
        =
        2^{-1/2}e(-1/8)
        \frac{h}{m^{3/2}}
        e\!\left(\frac{h^2}{4m}\right)
        +O\left(           
            \frac{n^{3/2}}{m^{3/2} \omega^2}
        \right).
\]
\end{lemma}

\begin{rem}
    The range in \eqref{eq:stationary range} is deliberately chosen slightly
wider than the support of $W$, so as to leave some room for the subsequent
non-stationary phase estimates.
\end{rem}

\begin{proof}
We first suppose $\sqrt{\frac{n}{m}} \leq \omega$. Then we can invoke \cite[Lemma 5.5.6]{Huxley1}.

Write
\[
        f(y)=hy-my^2,\qquad g(y)=A_{\omega,n}(y)=2yW_{\omega,n}(y^2).
\]
Then
\[
        I_m(h,n)=\int_{\mathbb R}g(y)e(f(y))\,dy .
\]
The phase has a unique stationary point
\[
        y_m=\frac{h}{2m},
\]
and
\[
        f(y_m)=\frac{h^2}{4m},\qquad f''(y)=-2m .
\]

Choose fixed constants 
\begin{equation}
\label{eq:c1_2}
\begin{split}
0 < c_1 < \frac{1}{2\sqrt{2}} \quad \text{and} \quad c_2 > 3/2, \\
\text{say } c_1=\frac{1}{3\sqrt{2}},\ c_2=2,
\end{split}
\end{equation}
and set
\[
\alpha=c_1\sqrt n,\qquad \beta=c_2\sqrt n.
\]
Then $\operatorname{supp}(g)\subseteq [\alpha,\beta]$, and $g(\alpha)=g(\beta)=0$.

We first record derivative bounds for $g$. Since
\[
        g(y)=2yW_{\omega,n}(y^2),
\]
we have, on the support of $g$,
\[
        g(y)\ll \sqrt n .
\]
Also
\[
        g'(y)
        =
        2W_{\omega,n}(y^2)+4y^2W'_{\omega,n}(y^2).
\]
By Lemma~\ref{lem:w def},
\[
        W'_{\omega,n}(x)\ll \omega^{-1}.
\]
Since $y^2\asymp n$ on the support of $g$, we get
\[
        g'(y)\ll 1+\frac{n}{\omega}\ll \frac{n}{\omega}.
\]
Similarly,
\[
        g''(y)
        =
        12yW'_{\omega,n}(y^2)+8y^3W''_{\omega,n}(y^2),
\]
and hence
\[
        g''(y)\ll \frac{\sqrt n}{\omega}+\frac{n^{3/2}}{\omega^2}
        \ll \frac{n^{3/2}}{\omega^2}.
\]

We now apply \cite[Lemma 5.5.6]{Huxley1} to $-f$, and then take the complex
conjugate. This accounts for the factor $e(-1/8)$ in the main term.

We need to choose proper parameters $U, N, M, T$ in \cite[Lemma 5.5.6]{Huxley1}. On the support of $g$,
we have $y\asymp \sqrt n$. By Lemma~\ref{lem:w def},
\[
        W_{\omega,n}^{(j)}(x)\ll_j \omega^{-j}.
\]
Hence
\[
        g(y)\ll \sqrt n,\qquad
        g'(y)\ll \frac{n}{\omega},\qquad
        g''(y)\ll \frac{n^{3/2}}{\omega^2},\qquad
        g^{(3)}(y)\ll \frac{n^2}{\omega^3}.
\]
Thus the amplitude conditions hold with 
\[      U = \sqrt n,\qquad N = \frac{\omega}{\sqrt n}.
\]
For the phase we can take
\[
        M = \sqrt n,\qquad T = mn.
\]

Therefore \cite[Lemma 5.5.6]{Huxley1} gives
\[
\begin{aligned}
        I_m(h,n)
        &=
\frac{g(y_m)e(f(y_m)-1/8)}{\sqrt{2m}}        \\
        &\quad
        +O\left(
            \frac{M^4U}{T^2}
            \left(1+\frac{M}{N}\right)^2
            \left(
                \frac{1}{(y_m-\alpha)^3}
                +
                \frac{1}{(\beta-y_m)^3}
            \right)
        \right)                               \\
        &\quad
        +O\left(  \frac{MU}{T^{3/2}}
            \left(1+\frac{M}{N}\right)^2
        \right).
\end{aligned}
\]

Substituting
\[
        M=\sqrt n,\qquad N=\frac{\omega}{\sqrt n},\qquad
        T=mn,\qquad U=\sqrt n
\]
we obtain
\[
        \frac{M}{N}=\frac{n}{\omega}.
\]
Moreover, by our choice \eqref{eq:c1_2} of the constants $\alpha,\beta$ and by the
assumption \eqref{eq:stationary range}, we have
\[
        y_m-\alpha\asymp \sqrt n,\qquad
        \beta-y_m\asymp \sqrt n .
\]
Hence the first error term is
\[
\begin{aligned}
&\frac{M^4U}{T^2}
\left(1+\frac{M}{N}\right)^2
\left(
        \frac{1}{(y_m-\alpha)^3}
        +
        \frac{1}{(\beta-y_m)^3}
\right)        \\
&\qquad\ll
\frac{n^{5/2}}{m^2n^2}
        \left(\frac{n}{\omega}\right)^2
        \frac{1}{n^{3/2}}
\ll
        \frac{n}{m^2 \omega^2}.
\end{aligned}
\]
Similarly, the second error term is
\[
\begin{aligned}
        \frac{MU}{T^{3/2}}
        \left(1+\frac{M}{N}\right)^2
        \ll
        \frac{n}{(mn)^{3/2}}
        \left(\frac{n}{\omega}\right)^2
    \ll
        \frac{n^{3/2}}{m^{3/2} \omega^2}.
\end{aligned}
\]
Therefore
\[
        I_m(h,n)
        =
        \frac{g(y_m)e(f(y_m)-1/8)}{\sqrt{2m}}
        +
        O\left(
            \frac{n^{3/2}}{m^{3/2} \omega^2}
        \right).
\]
Since
\[
        f(y_m)=\frac{h^2}{4m}
\]
and
\[
\begin{aligned}
        g(y_m)
        =
        2y_m W_{\omega,n}(y_m^2)
        =
        \frac{h}{m}
        W_{\omega,n}\left(\frac{h^2}{4m^2}\right),
\end{aligned}
\]
we get
\[
        I_m(h,n)
        =
        e\left(\frac{h^2}{4m}-\frac18\right)
        \frac{h}{\sqrt{2}m^{3/2}}
        W_{\omega,n}\left(\frac{h^2}{4m^2}\right)
        +
        O\left(
            \frac{n^{3/2}}{m^{3/2} \omega^2}
        \right).
\]

It remains to handle the case of $1 \le \omega \leq \sqrt{\frac{n}{m}}$. We shall use only a crude
second derivative estimate.  Since
\[
f(y)=hy-my^2=\frac{h^2}{4m}-m(y-y_m)^2,\qquad y_m=\frac{h}{2m}.
\]
Hence, for every interval $J,$
\[
\int_J e(f(y))\,dy
=
e\left(\frac{h^2}{4m}\right)m^{-1/2}
\int_{\sqrt m(J-y_m)} e(-u^2)\,du.
\]
Since
\[
\sup_{A<B}\left|\int_A^B e(-u^2)\,du\right|\ll 1,
\]
we have
\[
\int_J e(f(y))\,dy\ll m^{-1/2}
\]
uniformly in $J$.

Fix any real number $a$, for instance the left endpoint of the support of $g$, and put
\[
F(t)=\int_a^t e(f(y))\,dy.
\]
Then
$\|F\|_\infty\ll m^{-1/2}$.  By partial summation,
\[
I_m(h,n)
=
\int_{-\infty}^{\infty} g(y)e(f(y))\,dy
=
[g(y)F(y)]_{-\infty}^{\infty}
-
\int_{-\infty}^{\infty} F(y)g'(y)\,dy.
\]
The boundary term vanishes.  Moreover,
\[
g'(y)
=
2W_{\omega,n}(y^2)+4y^2W'_{\omega,n}(y^2).
\]
Since $y\asymp\sqrt n$ on the support and $\int |W'_{\omega,n}(x)|\,dx\ll 1$, we get
\[
\int |g'(y)|\,dy\ll \sqrt n.
\]
Therefore
\[
I_m(h,n)\ll \frac{\sqrt n}{\sqrt m},
\]
which is dominated by the error term $O\left(      
            \frac{n^{3/2}}{m^{3/2} \omega^2}
        \right)$ under $\omega \leq \sqrt{\frac{n}{m}}.$

This completes the proof.

\end{proof}

 By Lemma \ref{lem:Im-stationary}, for $m \in \left[\frac{1}{3}\frac{h}{\sqrt{n}}, \sqrt{2}\frac{h}{\sqrt{n}}\right]$ we have
\begin{equation*}
\begin{split}
    I_m(h, n)
&= C\,e\!\left(\frac{h^2}{4m}\right)
\frac{h}{m^{3/2}} W_{\omega,n}\!\left(\frac{h^2}{4m^2}\right)
        +O\left(
            \frac{n^{3/2}}{m^{3/2} \omega^2}
        \right),
\end{split}
\end{equation*}
for some absolute constant $C$.

We next estimate the integrals $I_m(h,n)$ outside the stationary range. First we introduce an auxiliary lemma.

\begin{lem}[Integration by parts]
\label{lemma:int_by_parts}
Let $G\in C_c^\infty(\mathbb R)$, and let $\phi$ be a real-valued
smooth phase such that $\phi'$ does not vanish on the support of $G$.
Assume moreover that $\phi''$ is constant. Then, for every integer
$J\geq 1$,
\[
\left|
\int_{\mathbb R}G(x)e(\phi(x))\,dx
\right|
\ll_J
\sum_{a+b=J}
\int_{\mathbb R}
\frac{|G^{(a)}(x)|\,|\phi''|^b}{|\phi'(x)|^{a+2b}}\,dx .
\]
\end{lem}

\begin{proof}
We have
\[
e(\phi(x))
=
\frac{1}{2\pi i\phi'(x)}
\frac{d}{dx}e(\phi(x)).
\]
Integrating by parts and using the compact support of $G$, we get
\[
\int_{\mathbb R}G(x)e(\phi(x))\,dx
=
-
\int_{\mathbb R}
\frac{d}{dx}
\left(
\frac{G(x)}{2\pi i\phi'(x)}
\right)
e(\phi(x))\,dx .
\]
Hence, if we define
\[
\mathcal L G
=
\frac{d}{dx}
\left(
\frac{G}{2\pi i\phi'}
\right),
\]
then repeated integration by parts gives
\[
\int_{\mathbb R}G(x)e(\phi(x))\,dx
=
(-1)^J
\int_{\mathbb R}\mathcal L^J G(x)e(\phi(x))\,dx .
\]

It remains to bound $\mathcal L^J G$. We claim that $\mathcal L^J G$
is a finite linear combination, with coefficients depending only on
$J$, of terms of the form
\[
\frac{G^{(a)}(x)(\phi'')^b}{(\phi'(x))^{a+2b}},
\qquad a+b=J.
\]
This is proved by induction on $J$. For $J=1$,
\[
\mathcal L G
=
\frac{1}{2\pi i}
\left(
\frac{G'}{\phi'}
-
\frac{G\phi''}{(\phi')^2}
\right),
\]
which is precisely the required form.

Suppose the claim holds after $J$ applications of $\mathcal L$. Consider
one of the resulting terms,
\[
\frac{G^{(a)}(x)(\phi'')^b}{(\phi'(x))^{a+2b}},
\qquad a+b=J.
\]
Applying $\mathcal L$ once more means dividing this term by
$2\pi i\phi'(x)$ and then differentiating. Up to a constant depending
only on $J$, we must differentiate
\[
\frac{G^{(a)}(x)(\phi'')^b}{(\phi'(x))^{a+2b+1}}.
\]
Since $\phi''$ is constant, the derivative can only hit either
$G^{(a)}(x)$ or the power of $\phi'(x)$ in the denominator. If it hits
$G^{(a)}(x)$, we get a term of the form
\[
\frac{G^{(a+1)}(x)(\phi'')^b}{(\phi'(x))^{a+2b+1}}
=
\frac{G^{(a+1)}(x)(\phi'')^b}
     {(\phi'(x))^{(a+1)+2b}}.
\]
If it hits $(\phi'(x))^{-(a+2b+1)}$, then one additional factor of
$\phi''$ appears, and we get a term of the form
\[
\frac{G^{(a)}(x)(\phi'')^{b+1}}{(\phi'(x))^{a+2b+2}}
=
\frac{G^{(a)}(x)(\phi'')^{b+1}}
     {(\phi'(x))^{a+2(b+1)}}.
\]
In both cases the new indices have sum $J+1$. This proves the
induction claim.

Taking absolute values in
\[
\int_{\mathbb R}G(x)e(\phi(x))\,dx
=
(-1)^J
\int_{\mathbb R}\mathcal L^J G(x)e(\phi(x))\,dx
\]
and using the above description of $\mathcal L^J G$, we obtain
\[
\left|
\int_{\mathbb R}G(x)e(\phi(x))\,dx
\right|
\ll_J
\sum_{a+b=J}
\int_{\mathbb R}
\frac{|G^{(a)}(x)|\,|\phi''|^b}{|\phi'(x)|^{a+2b}}\,dx ,
\]
as required.
\end{proof}

\begin{lemma}[Non-stationary phase]
\label{lem:Im-nonstat}
Let $1\leq \omega \leq n/10$, 
$A_{\omega,n}(y)$ is given by \eqref{eq:A}
and
\[
I_m(h,n)=\int_{\mathbb R}A_{\omega,n}(y)e(hy-my^2)\,dy.
\]
Then, for every integer $h\geq 1$ and every fixed integer $K\geq 1$,
we have
\begin{equation}
\label{eq:sum_m_1}
    \sum_{1\leq m\leq h/(3\sqrt n)}
|I_m(h,n)|
\ll_K
\min\left\{1,\left(\frac{\sqrt n}{\omega h}\right)^K\right\},
\end{equation}
\begin{equation}
\label{eq:sum_m_2}
    \sum_{m\geq \sqrt 2 h/\sqrt n}
|I_m(h,n)|
\ll_K
\frac{1}{\omega^K(1+h/\sqrt n)^K},
\end{equation}
\begin{equation}
\label{eq:sum_m_minus}
    \sum_{m<0}|I_m(h,n)|
\ll_K
\frac{1}{\omega^K(1+h/\sqrt n)^K},
\end{equation}
and
\begin{equation}
\label{eq:sum_m_0}
    |I_0(h,n)|
\ll_K
\frac{\sqrt n}{h}
\min\left\{1,\left(\frac{\sqrt n}{\omega h}\right)^K\right\}.
\end{equation}
Consequently,
\begin{equation}
\label{eq:exc_m}
\sum_{\substack{m\in\mathbb Z\\
m\notin [h/(3\sqrt n),\,\sqrt 2 h/\sqrt n]}}
|I_m(h,n)|
\ll_K
\left(1+\frac{\sqrt n}{h}\right)
\min\left\{1,\left(\frac{\sqrt n}{\omega h}\right)^K\right\}.
\end{equation}

\end{lemma}

\begin{proof}
We first record elementary $L^1$ bounds for the amplitude $A_{\omega,n}$. Since
$W_{\omega,n}$ is supported in $[n/2,n]$, the function $A_{\omega,n}$ is
supported in
\[
\sqrt{n/2}\leq y\leq \sqrt n .
\]
Moreover,
\[
\|A_{\omega,n}\|_1\ll n,
\]
and, for every integer $j\geq 1$,
\begin{equation}
\label{eq:L1_bound}
    \|A_{\omega,n}^{(j)}\|_1
\ll_j
\frac{n^{j/2}}{\omega^{j-1}}.
\end{equation}
Indeed, for $j=1$,
\[
A_{\omega,n}'(y)
=
2W_{\omega,n}(y^2)+4y^2W_{\omega,n}'(y^2).
\]
Since
\[
\int |2W_{\omega,n}(y^2)|\,dy\ll \sqrt n,
\]
and, after the change of variables $x=y^2$,
\[
\int 4y^2|W_{\omega,n}'(y^2)|\,dy
=
\int_{n/2}^n 2\sqrt x\,|W_{\omega,n}'(x)|\,dx
\ll
\sqrt n.
\]
Thus $\|A_{\omega,n}'\|_1\ll \sqrt n$. For $j\geq 2$, the chain rule
shows that $A_{\omega,n}^{(j)}$ is a finite linear combination of terms
of the form
\[
y^{2\ell+1-j}W_{\omega,n}^{(\ell)}(y^2),
\qquad 1\leq \ell\leq j .
\]
Changing variables $x=y^2$ and using
\[
\|W_{\omega,n}^{(\ell)}\|_1\ll_\ell \omega^{1-\ell}
\qquad (\ell\geq 1),
\]
we get
\[
\int
\left|
y^{2\ell+1-j}W_{\omega,n}^{(\ell)}(y^2)
\right|\,dy
\ll_j
n^{\ell-j/2} \omega^{1-\ell}
\ll_j
\frac{n^{j/2}}{\omega^{j-1}},
\]
since $\omega \leq n,\  \ell \le j$. This proves the claimed $L^1$ bounds \eqref{eq:L1_bound} for $A_{\omega,n}^{(j)}$.

Now we consider
\[
1\leq m\leq \frac{h}{3\sqrt n}.
\]
Put
\[
f_m(y)=hy-my^2.
\]
On the support of $A_{\omega,n}$,
\[
|f_m'(y)|
=
|h-2my|
\geq h-2m\sqrt n
\geq \frac{h}{3}.
\]
Also
\[
|f_m''(y)|=2m\leq \frac{2h}{3\sqrt n}.
\]
With one integration by parts, using
\[
\|A_{\omega,n}'\|_1\ll \sqrt n,
\qquad
\|A_{\omega,n}\|_1\ll n,
\]
we obtain
\[
|I_m(h,n)|
\ll
\frac{\sqrt n}{h}
+
\frac{mn}{h^2}
\ll
\frac{\sqrt n}{h}.
\]
Therefore
\[
\sum_{1\leq m\leq h/(3\sqrt n)}
|I_m(h,n)|
\ll 1.
\]

Now Lemma \ref{lemma:int_by_parts} with $J=K+1$ gives
\begin{equation}
\label{eq:Im_pre}
    |I_m(h,n)|
\ll_J
\sum_{a+b=J}
\frac{\|A_{\omega,n}^{(a)}\|_1m^b}{h^{a+2b}}.
\end{equation}
Substituting $m\leq \frac{h}{3\sqrt n}$ and the $L^1$ bounds \eqref{eq:L1_bound} for
$A_{\omega,n}^{(a)}$ into \eqref{eq:Im_pre}, we get
\[
|I_m(h,n)|
\ll_J
\frac{n^{J/2}}{\omega^{J-1}h^J}.
\]
Hence
\[
\sum_{1\leq m\leq \frac{h}{3\sqrt n}}
|I_m(h,n)|
\ll_J
\frac{h}{\sqrt n}
\frac{n^{J/2}}{\omega^{J-1}h^J}
=
\left(\frac{\sqrt n}{\omega h}\right)^{J-1}.
\]
Since $J=K+1$, this gives
\[
\sum_{1\leq m\leq \frac{h}{3\sqrt n}}
|I_m(h,n)|
\ll_K
\left(\frac{\sqrt n}{\omega h}\right)^K.
\]
Combining this estimate with the previous $O(1)$ estimate proves \eqref{eq:sum_m_1}.

We next consider the range
\[
m\geq \frac{\sqrt 2 h}{\sqrt n}.
\]
Make the change of variables $y=\sqrt n\,t$. Then
\[
I_m(h,n)
=
n\int_{\mathbb R}B_{\omega,n}(t)e(\Phi_m(t))\,dt,
\]
where
\[
B_{\omega,n}(t)=2tW_{\omega,n}(nt^2),
\qquad
\Phi_m(t)=h\sqrt n\,t-mnt^2.
\]
The function $B_{\omega,n}$ is supported in $[\frac{\sqrt{2}}{2}, 1]$. 

Moreover,
\[
\|B_{\omega,n}\|_1\ll 1,
\]
and, for every $j\geq 1$,
\begin{equation}
\label{eq:Bj_L1}
    \|B_{\omega,n}^{(j)}\|_1
\ll_j
\left(\frac n\omega\right)^{j-1}.
\end{equation}
The proof of \eqref{eq:Bj_L1} is similar to that of \eqref{eq:L1_bound} and we omit it here.

On the support of $B_{\omega,n}$, the assumption
$m\geq \frac{\sqrt 2 h}{\sqrt n}$ gives
\[
|\Phi_m'(t)|
=
|h\sqrt n-2mnt|
\asymp mn,
\qquad
|\Phi_m''(t)|=2mn.
\]
Invoking Lemma \ref{lemma:int_by_parts} with $J=K+1$, we get
\[
|I_m(h,n)|
\ll_J
n\sum_{a+b=J}
\frac{\|B_{\omega,n}^{(a)}\|_1(mn)^b}{(mn)^{a+2b}}
\ll_J
\frac{1}{m^J \omega^{J-1}}.
\]
Therefore
\begin{equation*}
\begin{aligned}
    \sum_{m \geq \frac{\sqrt{2} h}{\sqrt{n}}} |I_m(h,n)| 
    &\ll_{J} \frac{1}{\omega^{J-1}} \sum_{m \geq \max\{1, \sqrt{2} h / \sqrt{n}\}} \frac{1}{m^J} \\
    &\ll_{J} \frac{1}{\omega^{J-1} (1 + h / \sqrt{n})^{J-1}}.
\end{aligned}
\end{equation*}
Since $J=K+1$, this proves
\eqref{eq:sum_m_2}.

Now let $m<0$, and write $m=-r$ with $r\geq 1$. Then
\[
I_{-r}(h,n)
=
\int_{\mathbb R}A_{\omega,n}(y)e(hy+ry^2)\,dy.
\]
Put
\[
g_r(y)=hy+ry^2.
\]
On the support of $A_{\omega,n}$,
\[
|g_r'(y)|=|h+2ry| \asymp h+r\sqrt n,
\qquad
|g_r''(y)|=2r.
\]
Invoking Lemma \ref{lemma:int_by_parts},
\[
|I_{-r}(h,n)|
\ll_J
\sum_{a+b=J}
\frac{\|A_{\omega,n}^{(a)}\|_1r^b}{(h+r\sqrt n)^{a+2b}}.
\]
Since
\[
r\leq \frac{h+r\sqrt n}{\sqrt n},
\]
and recalling \eqref{eq:L1_bound},
\[
\|A_{\omega,n}^{(a)}\|_1
\ll_a
\begin{cases}
n, & a=0,\\
n^{a/2} \omega^{1-a}, & a\geq 1,
\end{cases}
\]
we get
\[
|I_{-r}(h,n)|
\ll_J
\frac{n^{J/2}}{\omega^{J-1}(h+r\sqrt n)^J}.
\]
Taking $J=K+1$ and summing over $r\geq 1$ gives
\[
\sum_{r\geq 1}|I_{-r}(h,n)|
\ll_K
\frac{n^{(K+1)/2}}{\omega^K}
\sum_{r\geq 1}\frac1{(h+r\sqrt n)^{K+1}}.
\]
Since
\[
h+r\sqrt n=\sqrt n\left(\frac h{\sqrt n}+r\right),
\]
we have
\[
\sum_{r\geq 1}\frac1{(h+r\sqrt n)^{K+1}}
\ll
\frac{1}{n^{(K+1)/2}(1+h/\sqrt n)^K}.
\]
Therefore we obtain \eqref{eq:sum_m_minus}.

It remains to treat $m=0$. In this case
\[
I_0(h,n)
=
\int_{\mathbb R}A_{\omega,n}(y)e(hy)\,dy.
\]
Lemma \ref{lemma:int_by_parts} gives
\[
|I_0(h,n)|
\ll_J
h^{-J}\|A_{\omega,n}^{(J)}\|_1.
\]
With $J=1$, this gives
\[
|I_0(h,n)|\ll \frac{\sqrt n}{h}.
\]
With $J=K+1$, it gives
\[
|I_0(h,n)|
\ll_K
\frac{n^{(K+1)/2}}{\omega^K h^{K+1}}
=
\frac{\sqrt n}{h}
\left(\frac{\sqrt n}{\omega h}\right)^K.
\]
Combining the two estimates, we obtain \eqref{eq:sum_m_0}.

Finally, putting $s=h/\sqrt {n}$, \eqref{eq:exc_m} is from
$$
\frac{1}{\omega^K(1+s)^K}
\ll
\left(1+\frac1s\right)
\min\left\{1,\frac{1}{(\omega s)^K}\right\}.
$$

We conclude the proof.

\end{proof}

\subsection{Dual approximation of the smoothed sum}

Recalling \eqref{eq:S_sum_I} and summing up Lemma \ref{lem:Im-stationary} and Lemma \ref{lem:Im-nonstat}, we obtain the following proposition.

\begin{prop}
\label{prop:dual-expansion-Sw}
Let $1\leq \omega \leq n/10$, and let $W_{\omega,n}$
be the smooth cut-off constructed in Lemma~\ref{lem:w def}. For $h\geq 1$, recall
\[
S_\omega(h,n)=\sum_{a\in\mathbb Z} W_{\omega,n}(a)e(h\sqrt a).
\]
Then, for every fixed integer $K\geq 1$,
$$
\begin{aligned}
  S_\omega(h,n) 
  &= C \sum_{\substack{m\in\mathbb Z \\ h/(3\sqrt n) \le m \le \sqrt2 h/\sqrt n}} \frac{h}{m^{3/2}} W_{\omega,n}\!\left(\frac{h^2}{4m^2}\right) e\!\left(\frac{h^2}{4m}\right) \\
  &\qquad \qquad \qquad + O_K(R_{\omega, K}(h,n)),
\end{aligned}
$$
where
\[
C=2^{-1/2}e(-1/8),
\]
and
$$
R_{\omega,K}(h,n)=
\frac{n^{7/4}}{\omega^2 h^{1/2}}
+
\left(1+\frac{\sqrt n}{h}\right)
\min\left\{1,\left(\frac{\sqrt n}{\omega h}\right)^K\right\}
.
$$
\end{prop}

\subsection{Weighted dual moments}

We reduce the smoothed moment to a dual large-sieve estimate.

Invoking Proposition \ref{prop:dual-expansion-Sw}, for $h\sim H$, one has
\begin{equation}
 \label{eq:smooth-poisson-higher-moment}
\begin{split}
    S_\omega(h,n)
  =&
  C
  \sum_{h/(3\sqrt n)\le m\le \sqrt 2 h/\sqrt n}
  \frac{h}{m^{3/2}}
  W_{\omega,n}\left(\frac{h^2}{4m^2}\right)
  e\left(\frac{h^2}{4m}\right)\\
  &\qquad  \qquad +
  O_K(R_{\omega,K}(H,n)),
\end{split}
\end{equation}
where
\begin{equation*}
  C=2^{-1/2}e(-1/8),
\end{equation*}
and
\begin{equation*}
  R_{\omega,K}(H,n)
  =
  \frac{n^{7/4}}{\omega^2H^{1/2}}
  +
  \left(1+\frac{\sqrt n}{H}\right)
  \min\left\{
    1,
    \left(\frac{\sqrt n}{\omega H}\right)^K
  \right\}.
\end{equation*}
Put
\begin{equation*}
  M=\frac{H}{\sqrt n}.
\end{equation*}
For $h\sim H$, the $m$-range in \eqref{eq:smooth-poisson-higher-moment}
is $m\asymp M$. Moreover
\begin{equation*}
  \frac{h}{m^{3/2}}
  \asymp
  \frac{H}{M^{3/2}}.
\end{equation*}
Define
\begin{equation*}
  B=\frac{H}{M^{3/2}}
  =
  n^{3/4}H^{-1/2}.
\end{equation*}
Absorbing the smooth amplitude into the weight
\begin{equation}
\label{eq:weight_V}
  V_{\omega;h,m}
  =
  \frac{h}{H}
  \left(\frac{M}{m}\right)^{3/2}
  W_{\omega,n}\left(\frac{h^2}{4m^2}\right),
\end{equation}
we may rewrite \eqref{eq:smooth-poisson-higher-moment} as
\begin{equation}
  \label{eq:dual-form-higher-moment}
  S_\omega(h,n)
  =
  CB
  \sum_{m\asymp M}
  V_{\omega;h,m}
  e\left(\frac{h^2}{4m}\right)
  +
  O_K(R_{\omega,K}(H,n)).
\end{equation}
The normalised weights satisfy
\begin{equation*}
  |V_{\omega;h,m}|\ll 1.
\end{equation*}

We define the dual moment
\begin{equation*}
  \mathcal{D}_{2r}(H,M;\omega)
  =
  \sum_{h\sim H}
  \left|
    \sum_{m\asymp M}
    V_{\omega;h,m}
    e\left(\frac{h^2}{4m}\right)
  \right|^{2r},
\end{equation*}
where $V_{\omega;h,m}$ are defined in \eqref{eq:weight_V}, and
$$
M=\frac{H}{\sqrt n}.
$$

We now isolate the part of the argument which is still missing in general. The following is the expected weighted dual large sieve estimate.

\begin{conj}[Weighted dual quadratic large sieve]
\label{conj:D2r}
Let $r\ge 1$ be fixed, let $\delta>0$, and suppose that $H\ge n^{1/2+\delta}$. Put $M=H/\sqrt n$. Then, for every $\varepsilon>0$,
$$
    \mathcal{D}_{2r}(H,M;\omega) \ll_{r,\varepsilon,\delta} H M^r n^\varepsilon \Phi_r(\omega),
$$
uniformly for $1\le \omega\le n/10$, where the dependence on $n$ is suppressed in the notation. In the ideal case one expects
\begin{equation}
 \label{eq:ideal-Phi}
  \Phi_r(\omega)\ll n^\varepsilon.
\end{equation}
\end{conj}

\begin{prop}[Conditional reduction]
\label{prop:DtoM}
Assume Conjecture \ref{conj:D2r}. If the ideal estimate \eqref{eq:ideal-Phi} holds, then for $H\ge n^{1/2+\delta}$,
$$
    M_{2r}(H,n)\ll_{r,\varepsilon,\delta} Hn^{r+\varepsilon}.
$$
\end{prop}

\begin{proof}
By \eqref{eq:dual-form-higher-moment} and the inequality $|x+y|^{2r}\ll_r |x|^{2r}+|y|^{2r}$,
$$
    M_{2r,\omega}(H,n) \ll_r B^{2r}\mathcal{D}_{2r}(H,M;\omega)+H R_{\omega,K}(H,n)^{2r}.
$$
Using Conjecture \ref{conj:D2r} and $B=H/M^{3/2}$, we get
$$
    B^{2r}H M^r = \left(\frac{H}{M^{3/2}}\right)^{2r}HM^r = H^{2r+1}M^{-2r} = Hn^r.
$$
Thus
$$
    M_{2r,\omega}(H,n) \ll_{r,\varepsilon,\delta} Hn^{r+\varepsilon}\Phi_r(\omega) + H R_{\omega,K}(H,n)^{2r}.
$$
Combining this with \eqref{eq:sharp-to-smooth-higher-moment}, we obtain
$$
    M_{2r}(H,n) \ll_{r,\varepsilon,\delta} Hn^{r+\varepsilon}\Phi_r(\omega) + H\omega^{2r} + H R_{\omega,K}(H,n)^{2r}.
$$
Now choose $\omega$ such that
\begin{equation}
    \label{eq:omega-admissible-range}
    n^{5/8}H^{-1/4}\ll \omega\le n^{1/2}.
\end{equation}
This interval is non-empty for $H\ge n^{1/2+\delta}$. For such $\omega$, the sharp cut-off contribution satisfies
$$
    H\omega^{2r}\le Hn^r,
$$
and the first part of $R_{\omega,K}$ is $O(n^{1/2})$. The remaining part of $R_{\omega,K}$ is $O(n^{-A})$, for any prescribed $A>0$, provided $K$ is chosen sufficiently large in terms of $A$ and $\delta$. Hence, under $\Phi_r(\omega)\ll n^\varepsilon$, we get
$$
    M_{2r}(H,n)\ll_{r,\varepsilon,\delta} Hn^{r+\varepsilon}.
$$
\end{proof}

Thus the remaining task is not to prove the moment
bound directly, but to prove a suitable estimate for the dual quantity
$\mathcal{D}_{2r}(H,M;\omega)$ with
$\Phi_r(\omega)\ll n^\varepsilon$ in the admissible range
\eqref{eq:omega-admissible-range}.

Recall that $V_{\omega;h,m}$ is defined by \eqref{eq:weight_V}. The next lemma shows that the weights $V_{\omega;h,m}$ have uniformly bounded variation.

\begin{lemma}[Bounded variation of the dual weights]
\label{lem:BV}
We have
\begin{equation}
\label{eq:bv-single-dual-weight}
    \sup_{m\asymp M} \left( \sup_{h\sim H}|V_{\omega;h,m}| + \operatorname{Var}_{h\sim H} V_{\omega;h,m} \right) \ll 1 .
\end{equation}
Consequently, for fixed $r\ge 1$, if
$$
    V_{\omega;\mathbf m,\mathbf n}(h) = \prod_{i=1}^r V_{\omega;h,m_i} \prod_{j=1}^r V_{\omega;h,n_j}, \qquad m_i,n_j\asymp M,
$$
then
\begin{equation}
\label{eq:bv-product-dual-weight}
    \sup_{\mathbf m,\mathbf n} \left( \sup_{h\sim H}|V_{\omega;\mathbf m,\mathbf n}(h)| + \operatorname{Var}_{h\sim H} V_{\omega;\mathbf m,\mathbf n}(h) \right) \ll_r 1 .
\end{equation}
\end{lemma}

\begin{proof}
Write
$$
    I_H=[H/2,H]
$$
for the dyadic interval $h\sim H$. We first prove \eqref{eq:bv-single-dual-weight}. Put
$$
    x_m(h)=\frac{h^2}{4m^2}, \qquad p_m(h)= \frac{h}{H} \left(\frac{M}{m}\right)^{3/2}.
$$
Then
$$
    V_{\omega;h,m} = p_m(h) W_{\omega,n}(x_m(h)).
$$

Since $m\asymp M$ and $h\sim H$, we have
$$
    \|p_m\|_{L^\infty(I_H)}\ll 1, \qquad \operatorname{Var}_{I_H} (p_m) \le \int_{I_H} |p_m'(h)|\,dh \ll 1 .
$$
It remains to control the variation of
$$
    W_{\omega,n}(x_m(h)).
$$
The function $x_m(h)$ is increasing on $I_H$, and
$$
    x_m'(h)=\frac{h}{2m^2}>0.
$$
Thus, by the change of variables $x=x_m(h)$,
$$
\begin{aligned}
    \operatorname{Var}_{I_H} W_{\omega,n}(x_m(h)) 
    &\le \int_{I_H} \left| W_{\omega,n}'(x_m(h))x_m'(h) \right|\,dh  \\
    &\le \int_{\mathbb R} |W_{\omega,n}'(x)|\,dx .
\end{aligned}
$$
Recall Lemma \ref{lem:w def}, we know
$$
    \int_{\mathbb R}|W_{\omega,n}'(x)|\,dx\ll 1.
$$
We conclude that
$$
    \left\|W_{\omega,n}(x_m(\cdot))\right\|_{L^\infty(I_H)} + \operatorname{Var}_{I_H} W_{\omega,n}(x_m(\cdot)) \ll 1 .
$$
Using
$$
    \operatorname{Var}_{I_H}(fg) \le \|f\|_{L^\infty(I_H)}\operatorname{Var}_{I_H}(g) + \|g\|_{L^\infty(I_H)}\operatorname{Var}_{I_H}(f),
$$
we obtain
$$
    \sup_{h\sim H}|V_{\omega;h,m}| + \operatorname{Var}_{h\sim H}V_{\omega;h,m} \ll 1,
$$
uniformly for $m\asymp M$. This proves \eqref{eq:bv-single-dual-weight}.

Now let
$$
    F(h)=V_{\omega;\mathbf m,\mathbf n}(h) = \prod_{\ell=1}^{2r}F_\ell(h),
$$
where each $F_\ell$ is one of the factors $V_{\omega;h,m_i}$ or $V_{\omega;h,n_j}$. By \eqref{eq:bv-single-dual-weight},
$$
    \|F_\ell\|_{L^\infty(I_H)} + \operatorname{Var}_{I_H}F_\ell \ll 1
$$
uniformly in the corresponding index. The elementary product inequality for functions of bounded variation gives
$$
    \operatorname{Var}_{I_H}\left(\prod_{\ell=1}^{2r}F_\ell\right) \le \sum_{\ell=1}^{2r} \operatorname{Var}_{I_H}(F_\ell) \prod_{k\ne \ell} \|F_k\|_{L^\infty(I_H)} .
$$
Therefore
$$
    \sup_{h\sim H}|V_{\omega;\mathbf m,\mathbf n}(h)| + \operatorname{Var}_{h\sim H}V_{\omega;\mathbf m,\mathbf n}(h) \ll_r 1,
$$
uniformly for all $m_i,n_j\asymp M$. This proves \eqref{eq:bv-product-dual-weight}.

\end{proof}

\begin{lem}[Partial summation reduction]
Let \(r\geq 1\) be fixed,
\begin{equation*}
  \mathbf m=(m_1,\ldots,m_r),
  \qquad
  \mathbf n=(n_1,\ldots,n_r).
\end{equation*}
Put
\begin{equation*}
  V_{\omega;\mathbf m,\mathbf n}(h)
  =
  \prod_{i=1}^r V_{\omega;h,m_i}
  \prod_{j=1}^r \overline{V_{\omega;h,n_j}},
\end{equation*}
and
\begin{equation*}
    \Delta(\mathbf{m}, \mathbf{n}) = \sum_{i=1}^r \frac{1}{m_i} - \sum_{j=1}^r \frac{1}{n_j}.
\end{equation*}
Then
\begin{equation}
\label{eq:BV-partial-summation-offdiag}
\begin{split}
&\left| \sum_{h\sim H} V_{\omega;\mathbf{m},\mathbf{n}}(h) e\left(\frac{h^2}{4}\Delta(\mathbf{m},\mathbf{n})\right) \right| \\
&\qquad \ll_{r} \sup_{J\subseteq [H/2,H]} \left| \sum_{h\in J} e\left(\frac{h^2}{4}\Delta(\mathbf{m},\mathbf{n})\right) \right|.
\end{split}
\end{equation}
where \(J\) runs over intervals with integer endpoints.
\end{lem}

\begin{proof}
This is a direct consequence of Lemma \ref{lem:BV} and partial
summation.
\end{proof}

\subsection{Proof of the second-moment theorem}

We first prove the weighted dual quadratic large-sieve estimate in the case $r=1$. This is the input needed to obtain the expected second moment bound in the smoothing range $n^{1/2+\delta} \ll H \ll n$.

\begin{prop}[The case $r=1$ of the weighted dual quadratic large sieve]
\label{prop:dual-large-sieve-r1}
Let $n^{1/2}\le H\le n$, and put
$$
    M=\frac{H}{\sqrt n}.
$$
Let $V_{\omega;h,m}$ be defined by \eqref{eq:weight_V}, and let
$$
    \mathcal{D}_2(H,M;\omega) = \sum_{h\sim H} \left| \sum_{m\asymp M} V_{\omega;h,m} e\left(\frac{h^2}{4m}\right) \right|^2 .
$$
Then, for every fixed $\varepsilon>0$,
$$
    \mathcal{D}_2(H,M;\omega) \ll_\varepsilon HM n^\varepsilon ,
$$
uniformly in $1\le \omega\le n/10$.
\end{prop}

\begin{proof}
Expanding the square, we get
$$
    \mathcal{D}_2(H,M;\omega) = \sum_{m_1,m_2\asymp M} \sum_{h\sim H} V_{\omega;h,m_1} V_{\omega;h,m_2} e\left( \frac{h^2}{4} \left( \frac1{m_1}-\frac1{m_2} \right) \right).
$$
The diagonal contribution $m_1=m_2$ is
$$
    \ll \sum_{m\asymp M} \sum_{h\sim H} |V_{\omega;h,m}|^2 \ll HM.
$$

It remains to treat $m_1\ne m_2$. Put
$$
    \alpha = \frac14 \left( \frac1{m_1}-\frac1{m_2} \right).
$$
By \eqref{eq:BV-partial-summation-offdiag},
$$
\begin{aligned}
    \left| \sum_{h\sim H} V_{\omega;h,m_1} V_{\omega;h,m_2} e(\alpha h^2) \right| &\ll \sup_{J\subseteq I_H} \left| \sum_{h\in J} e(\alpha h^2) \right|,
\end{aligned}
$$
where $I_H$ denotes the dyadic interval $h\sim H$, and $J$ runs over intervals with integer endpoints.

We now estimate the last quadratic exponential sum. Write
$$
    m_1=gu,\qquad m_2=gv,\qquad (u,v)=1,
$$
so that $u,v\asymp U:=M/g$. Then
$$
    \alpha = \frac{v-u}{4guv}.
$$
Let
$$
    d=(v-u,4g).
$$
Since $(u,v)=1$, we have $(v-u,uv)=1$. Therefore, after reducing $\alpha$ to lowest terms, its denominator is
$$
    q=\frac{4guv}{d}.
$$
In particular,
$$
    q\asymp \frac{gU^2}{d} = \frac{M^2}{gd}.
$$
The quadratic Weyl bound \cite[Theorem 8.1]{IwKow} gives, uniformly for intervals $J\subseteq I_H$,
\begin{equation*}
    \sum_{h\in J}e(\alpha h^2) \ll Hq^{-1/2}+q^{1/2}\log(2 q).
\end{equation*}
Since $H\le n$, we have $q\ll n$, and hence the logarithm may be absorbed into $n^\varepsilon$. Thus
$$
    \sum_{h\in J}e(\alpha h^2) \ll_\varepsilon n^\varepsilon \left( H\frac{\sqrt{gd}}{M} + \frac{M}{\sqrt{gd}} \right).
$$

For fixed $g$ and $d\mid 4g$, the number of pairs $u,v\asymp U$ with $v\equiv u\pmod d$ is
$$
    \ll \frac{U^2}{d}+U.
$$
Therefore the total off-diagonal contribution is
\begin{equation*}
\begin{aligned}
    \mathcal{D}_2^{\mathrm{off}}(H,M;\omega)
    &\ll_{\varepsilon} n^{\varepsilon}
      \sum_{g\ll M} \sum_{d\mid 4g}
      \left( H\frac{\sqrt{gd}}{M} + \frac{M}{\sqrt{gd}} \right)\left( \frac{U^2}{d}+U \right) \\
    &\ll n^{\varepsilon} \sum_{g\ll M} \sum_{d\mid 4g}
      \left(
      \frac{HM}{g^{3/2}d^{1/2}}
      +H\left(\frac{d}{g}\right)^{1/2}
      \right.\\
    &\hspace{45mm}\left.
      +\frac{M^3}{g^{5/2}d^{3/2}}
      +\frac{M^2}{g^{3/2}d^{1/2}}
      \right).
\end{aligned}
\end{equation*}
Using the divisor estimates
$$
    \sum_{d\mid 4g}d^{-1/2}\ll_\varepsilon g^\varepsilon, \qquad \sum_{d\mid 4g}d^{-3/2}\ll_\varepsilon g^\varepsilon, \qquad \sum_{d\mid 4g}d^{1/2}\ll_\varepsilon g^{1/2+\varepsilon},
$$
we obtain
$$
    \mathcal{D}_2^{\mathrm{off}}(H,M;\omega) \ll_\varepsilon n^\varepsilon \left( HM+M^3\right).
$$
Since $M=H/\sqrt n$ and $H\le n$, we have
$$
    M^3\le HM.
$$
Hence
$$
    \mathcal{D}_2^{\mathrm{off}}(H,M;\omega) \ll_\varepsilon HM n^\varepsilon.
$$
\end{proof}

\begin{proof}[Proof of Theorem~\ref{thm:M2-main}]

Proposition~\ref{prop:dual-large-sieve-r1} verifies Conjecture~\ref{conj:D2r} with $r=1$ and $\Phi_1(\omega) \ll n^{\varepsilon}$. Hence Propositions~\ref{prop:DtoM} gives the desired bound for $n^{1/2+\delta} \le H \le n$.
If $H\ge n$, then Lemma~\ref{lem:M2Hn} gives
\[
    M_2(H,n)
    \ll_{\varepsilon}
    n^\varepsilon\left(Hn+n^2\log H\right)
    \ll_{\varepsilon}
    Hn^{1+\varepsilon},
\]
after increasing $\varepsilon$ slightly. This proves the theorem.

\end{proof}


\subsection{Proof of the fourth-moment theorem}

In this subsection we prove a fourth moment estimate for the dual quadratic sums. After inserting $M=H/\sqrt{n}$, this estimate gives the expected bound for $M_4(H,n)$ in the smoothing range
\begin{equation*}
    n^{1/2+\delta} \leq H \ll n^{2/3}.
\end{equation*}

We begin with a simple exact-energy estimate for reciprocal sums.

\begin{lemma}[Two-dimensional reciprocal energy]
\label{lem:reciprocal-exact-energy-two}
For every fixed $\varepsilon>0$, one has
$$
    E^{=}_2(M) := \#\left\{ m_1,m_2,n_1,n_2\sim M: \frac1{m_1}+\frac1{m_2} = \frac1{n_1}+\frac1{n_2} \right\} \ll_\varepsilon M^{2+\varepsilon}.
$$
\end{lemma}

\begin{proof}
This is the case $r=2$ of~\cite[Lemma~4.2]{ShparlinskiKloosterman}; the underlying estimate is due to Karatsuba~\cite[proof of Theorem~1]{Karatsuba}, and see also~\cite[Lemma~4]{BourgainGaraev}. Indeed, the dyadic solutions counted by $E^{=}_2(M)$ form a subset of the solutions with all four variables in $[1,M]$.
\end{proof}

Next we record the denominator sum which will be used to estimate the off-diagonal contribution.

\begin{lemma}[A reciprocal-denominator gcd sum]
\label{lem:reciprocal-denominator-gcd}
Let $M\geq 1$. For $a,b\sim M$, write
$$
    \theta(a,b)=\frac1a+\frac1b = \frac{u(a,b)}{v(a,b)}, \qquad (u(a,b),v(a,b))=1.
$$
Then, for every fixed $\varepsilon>0$,
$$
    \sum_{a,b,c,d\sim M} \frac{(v(a,b),v(c,d))}{v(a,b)^{1/2}v(c,d)^{1/2}} \ll_\varepsilon M^{2+\varepsilon}.
$$
\end{lemma}

\begin{proof}
First note that $v(a,b)\ll M^2$. Indeed, if
$$
    a=gx,\qquad b=gy,\qquad (x,y)=1,
$$
then
$$
    \theta(a,b)=\frac{x+y}{gxy},
$$
and hence the reduced denominator divides
$$
    gxy=\frac{ab}{(a,b)}\ll M^2.
$$

Using
$$
    (r,s)\leq \sum_{\ell\mid r,\ \ell\mid s}\ell,
$$
we get
$$
    \sum_{a,b,c,d\sim M} \frac{(v(a,b),v(c,d))}{v(a,b)^{1/2}v(c,d)^{1/2}} \leq \sum_{\ell\ll M^2}\ell A_\ell^2,
$$
where
$$
    A_\ell = \sum_{\substack{a,b\sim M\\ \ell\mid v(a,b)}} v(a,b)^{-1/2}.
$$
We claim that
$$
    A_\ell\ll_\varepsilon \frac{M^{1+\varepsilon}}{\ell}.
$$

To prove this, write
$$
    a=gx,\qquad b=gy,\qquad (x,y)=1.
$$
Put
$$
    t=(g,x+y),\qquad g=ts.
$$
Then
$$
    \theta(a,b) = \frac{x+y}{tsxy} = \frac{(x+y)/t}{sxy},
$$
and the last fraction is in lowest terms. Hence
$$
    v(a,b)=sxy.
$$
Moreover,
$$
    tsx\sim M,\qquad tsy\sim M.
$$
Ignoring the conditions $(x,y)=1$ and $t\mid x+y$ can only enlarge the sum. Therefore
$$
    A_\ell \leq \sum_{\substack{t,s,x,y\\ tsx\sim M,\ tsy\sim M\\ \ell\mid sxy}} (sxy)^{-1/2}.
$$
If $\ell\mid sxy$, then we may write
$$
    \ell=\ell_0\ell_1\ell_2, \qquad \ell_0\mid s,\quad \ell_1\mid x,\quad \ell_2\mid y.
$$
Thus
$$
    A_\ell \leq \sum_{\ell_0\ell_1\ell_2=\ell} \sum_{\substack{t,s,x,y\\ tsx\sim M,\ tsy\sim M\\ \ell_0\mid s,\ \ell_1\mid x,\ \ell_2\mid y}} (sxy)^{-1/2}.
$$
For fixed $t,s$, put
$$
    X=\frac{M}{ts}.
$$
Then $x,y\sim X$. We use
$$
    \sum_{\substack{x\sim X\\ \ell_1\mid x}}x^{-1/2} \ll \frac{X^{1/2}}{\ell_1},
$$
and similarly for $y$. Hence
$$
    A_\ell \ll \sum_{\ell_0\ell_1\ell_2=\ell} \sum_{\substack{t,s\\ ts\ll M\\ \ell_0\mid s}} s^{-1/2}\frac{M}{ts}\frac1{\ell_1\ell_2}.
$$
Therefore
$$
    A_\ell \ll M \sum_{\ell_0\ell_1\ell_2=\ell} \frac1{\ell_1\ell_2} \sum_{\substack{s\ll M\\ \ell_0\mid s}}s^{-3/2} \sum_{t\ll M/s}\frac1t.
$$
Since
$$
    \sum_{t\ll M/s}\frac1t\ll M^\varepsilon
$$
and
$$
    \sum_{\substack{s\ll M\\ \ell_0\mid s}}s^{-3/2} \ll \ell_0^{-3/2},
$$
we obtain
$$
    A_\ell \ll_\varepsilon M^{1+\varepsilon} \sum_{\ell_0\ell_1\ell_2=\ell} \frac1{\ell_1\ell_2\ell_0^{3/2}}.
$$
Note that
$$
    \sum_{\ell_0\ell_1\ell_2=\ell} \frac1{\ell_1\ell_2\ell_0^{3/2}}\ll_\varepsilon \ell^{-1+\varepsilon}.
$$
Since $\ell\ll M^2$, the factor $\ell^\varepsilon$ may be absorbed into $M^\varepsilon$. Hence
$$
    A_\ell\ll_\varepsilon \frac{M^{1+\varepsilon}}{\ell}.
$$

Consequently,
$$
    \sum_{\ell\ll M^2}\ell A_\ell^2 \ll_\varepsilon \sum_{\ell\ll M^2}\ell\frac{M^{2+\varepsilon}}{\ell^2} \ll_\varepsilon M^{2+\varepsilon}.
$$
This proves the lemma.
\end{proof}

We now prove the fourth dual moment estimate.

\begin{prop}
\label{prop:D4-dual-unconditional}
Recall that
$$
\mathcal{D}_4(H,M;\omega)=\sum_{h\sim H}|T_\omega(h)|^4,
$$
where
$$
    T_\omega(h) = \sum_{m\sim M} V_{\omega;h,m} e\left(\frac{h^2}{4m}\right).
$$
Uniformly in $H,M$ and $1\leq \omega\leq n/10$, one has
$$
    \mathcal{D}_4(H,M;\omega) \ll_\varepsilon HM^{2+\varepsilon}+M^{6+\varepsilon}.
$$
\end{prop}

\begin{proof}
Expanding the fourth power gives
$$
    \mathcal{D}_4(H,M;\omega) = \sum_{m_1,m_2,n_1,n_2\sim M} \sum_{h\sim H} V_{\omega;h,m_1}V_{\omega;h,m_2} V_{\omega;h,n_1}V_{\omega;h,n_2} e\left(\frac{h^2}{4}\Lambda\right),
$$
where
$$
    \Lambda = \frac1{m_1}+\frac1{m_2} -\frac1{n_1}-\frac1{n_2}.
$$
We split the contribution according as $\Lambda=0$ or $\Lambda\neq 0$.

The contribution of $\Lambda=0$ is bounded, by Lemma \ref{lem:reciprocal-exact-energy-two} and the uniform boundedness of the weights, by
$$
    \ll HE^{=}_2(M) \ll_\varepsilon HM^{2+\varepsilon}.
$$

It remains to estimate the off-diagonal contribution. 
By \eqref{eq:BV-partial-summation-offdiag},
\begin{equation*}
    \begin{split}
        &\left| \sum_{h\sim H} V_{\omega;h,m_1} V_{\omega;h,m_2} \overline{V_{\omega;h,n_1}} \overline{V_{\omega;h,n_2}} e\left( \frac{h^2}{4}\Lambda \right) \right| \\
        &\qquad \qquad \ll \sup_{J\subset [H/2,H]} \left| \sum_{h\in J} e\left( \frac{h^2}{4} \Lambda \right) \right|.
    \end{split}
\end{equation*}
By Lemma \ref{lem:reciprocal-denominator-gcd}, for a pair $a,b\sim M$, write
$$
    \theta(a,b)=\frac1a+\frac1b = \frac{u(a,b)}{v(a,b)}
$$
in lowest terms. We have already observed that
$$
    v(a,b)\ll M^2.
$$

Let
$$
    \theta(m_1,m_2)=\frac{u}{v}, \qquad \theta(n_1,n_2)=\frac{u'}{v'}
$$
be written in lowest terms. In the off-diagonal case,
$$
    \theta(m_1,m_2)\neq \theta(n_1,n_2).
$$
Let $Q$ be the denominator of
$$
  \Lambda =  \theta(m_1,m_2)-\theta(n_1,n_2) = \frac{u}{v}-\frac{u'}{v'}
$$
in lowest terms. Then
$$
    Q = \frac{vv'}{(uv'-u'v,vv')}.
$$
Moreover,
$$
    (uv'-u'v,vv')\leq (v,v')^2.
$$
Therefore
$$
    Q^{-1/2} \ll \frac{(v,v')}{(vv')^{1/2}}.
$$

By Lemma \ref{lem:reciprocal-exact-energy-two}, the diagonal has already been removed, so $Q\geq 1$. The harmless factor $1/4$ in the phase changes the denominator by at most an absolute factor. Thus, invoke the quadratic Weyl bound \cite[Theorem 8.1]{IwKow} again, uniformly for intervals $J\subset [H/2,H]$,
$$
    \sum_{h\in J}e\left(\frac{\Lambda h^2}{4}\right) \ll HQ^{-1/2}+Q^{1/2}\log(2Q).
$$
Since $Q\ll vv'\ll M^4$, we get
$$
    \sum_{h\in J}e\left(\frac{\Lambda h^2}{4}\right) \ll_\varepsilon H\frac{(v,v')}{(vv')^{1/2}} + M^{2+\varepsilon}.
$$

Therefore the off-diagonal contribution is
$$
    \ll_\varepsilon H\Sigma+M^{6+\varepsilon},
$$
where
$$
    \Sigma = \sum_{m_1,m_2,n_1,n_2\sim M} \frac{(v(m_1,m_2),v(n_1,n_2))}{v(m_1,m_2)^{1/2}v(n_1,n_2)^{1/2}}.
$$
By Lemma \ref{lem:reciprocal-denominator-gcd},
$$
    \Sigma\ll_\varepsilon M^{2+\varepsilon}.
$$
Hence the off-diagonal contribution is
$$
    \ll_\varepsilon HM^{2+\varepsilon}+M^{6+\varepsilon}.
$$
Combining this with the diagonal contribution proves the proposition.
\end{proof}

\begin{proof}[Proof of Theorem~\ref{thm:M4-main}]
Recall that
\[
    M=\frac{H}{\sqrt n}.
\]
By Proposition~\ref{prop:D4-dual-unconditional},
\[
    \mathcal{D}_4(H,M;\omega)
    \ll_{\varepsilon}
    HM^{2+\varepsilon}+M^{6+\varepsilon}.
\]
If $H\le n^{2/3}$, then
\[
    M^6
    =
    \left(\frac{H}{\sqrt n}\right)^6
    \le
    H\left(\frac{H}{\sqrt n}\right)^2
    =
    HM^2.
\]
Hence
\[
    \mathcal{D}_4(H,M;\omega)
    \ll_{\varepsilon}
    HM^{2+\varepsilon}.
\]
Applying Proposition~\ref{prop:DtoM} with $r=2$, we obtain
\[
    M_4(H,n)\ll_{\varepsilon,\delta} Hn^{2+\varepsilon}.
\]
This proves the theorem.
\end{proof}

\subsection{Higher moments: conjectural picture}

We conclude this section by recording the conjectural picture suggested
by the dual expansion in Proposition~\ref{prop:dual-expansion-Sw}. For
$h\sim H$, the main term is a dual $m$-sum of length
\[
    M\asymp \frac{H}{\sqrt n},
\]
and each summand has size
\[
    B=\frac{H}{M^{3/2}}=n^{3/4}H^{-1/2}.
\]
If the phases
\[
    e\left(\frac{h^2}{4m}\right)
\]
exhibit square-root cancellation in the $m$-aspect, then a typical
value of the dual sum should have size
\[
    BM^{1/2}=n^{1/2}.
\]
Thus one expects
\[
    M_{2r}(H,n)\ll_{r,\varepsilon,\delta} Hn^{r+\varepsilon},
    \qquad
    H\geq n^{1/2+\delta}.
\]

For $H\leq n^{1/2-\delta}$, the pointwise estimate
\eqref{eq:OptBound} gives the complementary bound
\[
    M_{2r}(H,n)
    \ll_{r,\varepsilon,\delta}
    n^{r+\varepsilon}H^{1-2r}.
\]
This leads to the following ideal moment conjecture.

\begin{conj}[Ideal moment bound]
Let $r\geq 1$ be fixed. For every fixed $\delta>0$ and every
$\varepsilon>0$, one expects
\[
    M_{2r}(H,n)
    \ll_{r,\varepsilon,\delta}
    n^\varepsilon
    \begin{cases}
        n^r H^{1-2r},
            & 1\leq H\leq n^{1/2-\delta},\\
        Hn^r,
            & H\geq n^{1/2+\delta}.
    \end{cases}
\]
\end{conj}

\section{Proof of the discrepancy bound}
\label{sec:dis}

\begin{proof}[Proof of Theorem~\ref{thm:discrk}]

We prove a slightly more general estimate. Let $k\geq 2$, and let
$(\kappa,\kappa)$ be an exponent pair with
\[
     0<\kappa\leq \frac{13}{84}.
\]
Set
\begin{equation}
\label{eq:rho-discrepancy}
    \rho=\rho_{k,\kappa}
    =
    \frac{(1-\kappa)k+2\kappa}
    {2(\kappa k+1-2\kappa)}.
\end{equation}
We shall prove
\[
    D_k(n)\leq n^{-\rho+o(1)}.
\]
Taking $\kappa=13/84$ then gives
\[
    \rho=\rho_k=\frac{71k+26}{26k+116},
    \qquad k\geq 2.
\]

We first record three elementary consequences of \eqref{eq:rho-discrepancy}.
For $k\geq 2$,
\begin{equation}
\label{eq:rho-first-check}
    \frac{k}{2}-\rho
    =
    \frac{\kappa(k-2)(k+1)}
    {2(\kappa k+1-2\kappa)}
    \geq 0.
\end{equation}
Also,
\begin{equation}
\label{eq:rho-second-check}
\begin{aligned}
    k-\rho-
    \left(
        1+\frac{k-2}{2(1-2\kappa)}
    \right)
    =
    \frac{
        \kappa(k-2)\bigl(k-1-(4k-2)\kappa\bigr)
    }
    {
        2(1-2\kappa)(\kappa k+1-2\kappa)
    }
    \geq 0,
\end{aligned}
\end{equation}
since $\kappa\leq 13/84<1/6\leq (k-1)/(4k-2)$. Finally,
\begin{equation}
\label{eq:rho-balance}
    \rho\bigl(\kappa(k-2)+1\bigr)
    +
    \frac{(1+\kappa)(k-2)}{2}
    +1
    =
    k.
\end{equation}

Put
\[
    T(h,n)=\sum_{1\le a\leq n}\e\left(h\sqrt a\right).
\]
Applying Corollary~\ref{cor:ET} to the $n^k$ points
\[
    \sqrt{a_1}+\cdots+\sqrt{a_k},
    \qquad 1\leq a_1,\ldots,a_k\leq n,
\]
and using the factorisation of the exponential sum, we get, for any
integer $H_0\geq 1$,
\begin{equation}
\label{eq:ET-discrepancy-final}
    D_k(n)
    \ll
    \frac1{H_0}
    +
    \frac1{n^k}
    \sum_{h\leq H_0}
    \frac{|T(h,n)|^k}{h}.
\end{equation}

Fix a small $\nu>0$, and choose
\begin{equation*}
    H_0 = \left\lfloor n^{\rho-\nu} \right\rfloor.
\end{equation*}
We shall prove that, uniformly for every dyadic $H\leq H_0$ and every dyadic $m\leq n$,
\begin{equation}
\label{eq:key-dyadic-discrepancy}
    \frac{1}{H} \sum_{h\sim H} \left| S(h,m) \right|^k \ll n^{k-\rho+O(\nu)+o(1)}.
\end{equation}
Indeed, decomposing $T(h,n)$ into dyadic blocks gives
\begin{equation*}
    \left| T(h,n) \right|^k \ll n^{o(1)} \sum_{m \leq n, \, \text{dyadic}} \left| S(h,m) \right|^k.
\end{equation*}
Hence dyadic summation over $h$ and $m$, together with \eqref{eq:key-dyadic-discrepancy}, gives
\begin{equation*}
    \sum_{h\leq H_0} \frac{\left| T(h,n) \right|^k}{h} \ll n^{k-\rho+O(\nu)+o(1)}.
\end{equation*}
Substituting this into~\eqref{eq:ET-discrepancy-final}, we obtain
\begin{equation*}
    D_k(n) \ll n^{-\rho+O(\nu)+o(1)}.
\end{equation*}
Since $\nu>0$ is arbitrary, this proves the desired discrepancy bound.

It remains to prove~\eqref{eq:key-dyadic-discrepancy}. We may assume $H,m\geq 2$. Put
\begin{equation*}
    \theta = \frac{1+2\kappa}{2(1-2\kappa)}.
\end{equation*}
After a harmless subdivision of dyadic intervals, we may assume that $h\sim H$ does not cross any of the transition points in~\eqref{eq:OptBound}.

First suppose that
\begin{equation*}
    H \leq m^{1/2+\nu}.
\end{equation*}
By Lemmas~\ref{lem:vdCorput-1} and~\ref{lem:vdCorput-2}, uniformly in this range, 
\begin{equation*}
    \left| S(h,m) \right| \ll m^{1/2+O(\nu)}.
\end{equation*}
Therefore
\begin{equation*}
    \frac{1}{H} \sum_{h\sim H} \left| S(h,m) \right|^k \ll m^{k/2+O(\nu)} \leq n^{k/2+O(\nu)} \leq n^{k-\rho+O(\nu)},
\end{equation*}
where the last inequality follows from~\eqref{eq:rho-first-check}.

Now assume that
\begin{equation*}
    H > m^{1/2+\nu}.
\end{equation*}
By Theorem~\ref{thm:M2-main}, applied with $m$ in place of $n$, we have
\begin{equation*}
    M_2(H,m) = \sum_{h\sim H} \left| S(h,m) \right|^2 \ll_{\nu} H m^{1+\nu}.
\end{equation*}
Thus, if $P(H,m)$ is a pointwise upper bound for $\left| S(h,m) \right|$ on $h\sim H$, then
\begin{equation}
\label{eq:moment-reduction-discrepancy}
\begin{split}
    \frac{1}{H} \sum_{h\sim H} \left| S(h,m) \right|^k 
    &\leq P(H,m)^{k-2} \frac{M_2(H,m)}{H} \\
    &\ll_{\nu} P(H,m)^{k-2} m^{1+\nu}.
\end{split}
\end{equation}
For $k=2$, the factor $P(H,m)^{k-2}$ is simply $1$.

Consider next the range
\begin{equation*}
    m^{1/2+\nu} < H \leq m^\theta.
\end{equation*}
By~\eqref{eq:OptBound},
\begin{equation*}
    P(H,m) \ll H^{1/2}m^{1/4}.
\end{equation*}
Hence~\eqref{eq:moment-reduction-discrepancy} gives
\begin{equation*}
\begin{aligned}
    \frac{1}{H} \sum_{h\sim H} \left| S(h,m) \right|^k 
    &\ll H^{(k-2)/2} m^{(k-2)/4+1+\nu}  \\
    &\leq m^{1+\frac{k-2}{2(1-2\kappa)}+\nu} \\
    &\leq n^{k-\rho+O(\nu)},
\end{aligned}
\end{equation*}
where the last step follows from~\eqref{eq:rho-second-check}.

Finally suppose that
\begin{equation*}
    H > m^\theta.
\end{equation*}
By~\eqref{eq:OptBound},
\begin{equation*}
    P(H,m) \ll H^{\kappa+o(1)} m^{1/2+\kappa/2}.
\end{equation*}
Using~\eqref{eq:moment-reduction-discrepancy}, we get
\begin{equation*}
\begin{aligned}
    \frac{1}{H} \sum_{h\sim H} \left| S(h,m) \right|^k 
    &\ll H^{\kappa(k-2)+o(1)} m^{(1+\kappa)(k-2)/2+1+\nu}  \\
    &\leq n^{\rho\kappa(k-2) + \frac{(1+\kappa)(k-2)}{2} + 1 + O(\nu) + o(1)} \\
    &= n^{k-\rho+O(\nu)+o(1)},
\end{aligned}
\end{equation*}
where we used $H \leq H_0 \leq n^{\rho-\nu}$, $m \leq n$, and the balancing identity~\eqref{eq:rho-balance}.

The three ranges prove~\eqref{eq:key-dyadic-discrepancy}, and hence
\begin{equation*}
    D_k(n) \leq n^{-\rho_{k,\kappa}+o(1)}.
\end{equation*}
Substituting Bourgain's exponent pair $\kappa=13/84$ gives
\[
    D_k(n) \leq n^{-\rho_k+o(1)},
    \qquad
    \rho_k=\frac{71k+26}{26k+116},
    \qquad k\ge 2,
\]
which is precisely Theorem~\ref{thm:discrk}.
\end{proof}

\section*{Acknowledgements}
The author is grateful to Igor Shparlinski for suggesting the problem
studied in this paper, and for many helpful discussions. The author
also thanks Stefan Steinerberger for useful remarks. This work was
supported by the China Scholarship Council.

\end{document}